\documentclass[11pt]{amsart}

\calclayout

\usepackage{graphicx}
\usepackage{amssymb, amsthm,mathrsfs}
\usepackage[all]{xy}
\usepackage[dvipsnames]{xcolor}
\usepackage{comment}
\usepackage{pinlabel}
\usepackage[inline]{enumitem}
\usepackage{todonotes}
\usepackage{thm-restate}
\usepackage{microtype}
\usepackage[T1]{fontenc}
\usepackage{enumitem}
\usepackage{comment}
\usepackage{tikz-cd}
\usetikzlibrary{positioning,fit,backgrounds}
\usepackage{ae,aecompl}
\usepackage{times}
\usepackage{float}
\usepackage{scalerel}
\usepackage{stmaryrd}
\usepackage{verbatim,amsmath,amsthm,amsfonts,amssymb,latexsym,graphicx,mathtools,color,mathabx,soul}
\usepackage[all]{xy}
\usepackage{graphicx}
\usepackage{caption}
\usepackage{subcaption}
\DeclareMathAlphabet{\mathpzc}{OT1}{pzc}{m}{it}
\usepackage[colorlinks,pagebackref,hypertexnames=false]{hyperref} 
\usepackage{hyperref}
\hypersetup{
  colorlinks = true,
  linkcolor  = black,
  citecolor=blue
} 

\usepackage[alphabetic,backrefs,msc-links]{amsrefs}
\usepackage{accents}

\graphicspath{./images/}
\usepackage{aliascnt}
\numberwithin{equation}{section}

\newtheorem{theorem}{Theorem}[section]
\newtheorem{prop}[theorem]{Proposition}
\newtheorem{lemma}[theorem]{Lemma}

\newtheorem{quest}{Question}
\newtheorem{cor}[theorem]{Corollary}
\newtheorem{conj}{Conjecture}
\newtheorem{post}{Postulate}

\theoremstyle{definition}
\newtheorem{definition}[theorem]{Definition}

\DeclareMathOperator{\coker}{coker}

\theoremstyle{remark}
\newtheorem{remark}[theorem]{Remark}

\begin{document}

\title{A Lefschetz decomposition over $\mathbb Z$, and applications}

\author[Analisa Faulkner Valiente]{Analisa Faulkner Valiente}
\address{Department of Mathematics, Barnard College, New York, NY 10025}
\email{acf2182@barnard.edu}

\author[Mike Miller Eismeier]{Mike Miller Eismeier}
\address{Department of Mathematics, University of Vermont, Burlington, VT 05405}
\email{Mike.Miller-Eismeier@uvm.edu}

\maketitle
\begin{abstract}
We discuss a `Lefschetz filtration' of $\Lambda^*(\mathbb Z^{2g})$ and prove its subquotients are isomorphic as $\text{Sp}(2g)$-modules to primitive subspaces $P^k(\mathbb Z^{2g})$. This gives a sort of integral version of the Lefschetz decomposition over $\mathbb C$. 

We present three applications: the precise failure of the Hard Lefschetz theorem for $\Lambda^*(\mathbb Z^{2g})$, a description of the $\text{Sp}(2g)$-module structure on the cohomology of integer Heisenberg groups, and a computation of the Heegaard Floer homology groups $HF^\infty(\Sigma_g \times S^1; \mathbb Z)$ as modules over the mapping class group. Our computation implies that $HF^\infty$ is not naturally isomorphic to Mark's `cup homology'.
\end{abstract}

\setcounter{tocdepth}{1}
\tableofcontents

\section{Introduction}
A compact $n$-dimensional K\"ahler manifold $(X, \omega, J)$ is a compact complex manifold of real dimension $2n$ equipped with a symplectic form $\omega$ for which $\omega(v, Jw)$ is a Riemannian metric on each tangent space. The symplectic form $\omega$ is closed, so defines a cohomology class $[\omega] \in H^2(X; \mathbb R)$. 

The celebrated \textit{hard Lefschetz theorem} \cite[page 122]{GriffithsHarris} asserts that the map \begin{equation}\label{eq:hLef}H^{n-i}(X;\mathbb R) \xrightarrow{\wedge [\omega]^i} H^{n+i}(X;\mathbb R)\end{equation} is an isomorphism. 

The hard Lefschetz theorem is closely related to what is alternately called the \emph{Hodge--LePage decomposition} \cite[Proposition 1.1]{BGG} or the \emph{Lefschetz decomposition}. Writing \[P^k(X) = \{x \in H^k(X;\mathbb C) \mid [\omega]^{n-k+1} \wedge x = 0\},\] for $0 \le k \le n$, understanding this vector space to be zero for $k > n$, the hard Lefschetz theorem implies the existence of a direct sum decomposition \begin{equation}\label{eq:H-LP}H^k(X;\mathbb C) = P^k(X) \oplus \omega P^{k-2}(X) \oplus  \omega^2 P^{k-4}(X) \oplus \cdots \end{equation}

If the K\"ahler form $\omega$ is integral, so that it may be considered as an element of \[H^2(X;\mathbb Z)/\text{Tors} = \text{im}\left(H^2(X;\mathbb Z) \to H^2(X; \mathbb R)\right),\] it makes sense to ask whether the map \eqref{eq:hLef} is an isomorphism over the integers, or whether the direct sum decomposition \eqref{eq:H-LP} holds at the level of abelian groups. Smooth complex projective varieties give examples of K\"ahler manifolds with integral K\"ahler form.

Both statements fail, badly, for multiple reasons: \begin{enumerate}[label=(\alph*)]
    \item The form $\omega^i$ is divisible; $\left[\frac{\omega^i}{i!}\right]$ is also an integer form. 
    \item While the subspaces $\frac{\omega^i}{i!} P^{k-2i}$ are independent, they fail to span $H^k(X;\mathbb Z)$ as soon as $k = 2$. 
    \item If $X$ is a smooth complex projective variety, then in the extreme case $i = n$, the map $\wedge \left[\frac{\omega^n}{n!}\right]$ is multiplication by $d$, where $[X] = d[\mathbb{CP}^n] \in H_{2n}(\mathbb{CP}^N;\mathbb Z)$. This is not an isomorphism unless $X$ is a projective subspace, and thus this map depends on the geometry of $X$. 
\end{enumerate}

To get a better handle on what the Hard Lefschetz theorem or the Lefschetz decomposition should mean over the integers, we investigate a special case: $\Lambda^*(\mathbb Z^{2g})$ equipped with the standard symplectic form. This example arises as the cohomology ring of the Jacobian variety $\text{Jac}(\Sigma_g) \cong (\mathbb C/\mathbb Z+i\mathbb Z)^g$, but (c) implies that the corresponding K\"ahler form does not arise from an embedding in $\mathbb{CP}^N$. We refer to this as the linear case, as it is analogous to studying the Lefschetz decomposition of $\Lambda^*(T_p M)$.

As a first observation, we may define a $\text{Sp}(2g)$-invariant Lefschetz filtration on the exterior powers \[F_r \Lambda^k(\mathbb Z^{2g}) = \{\alpha \in \Lambda^k(\mathbb Z^{2g}) \mid \omega^{g-k+r+1} \wedge \alpha = 0\};\] for $0 \le k \le g$, this filtration begins at $F_0 \Lambda^k = P^k(\mathbb Z^{2g})$, while for $k \ge g$ the first nonzero term is the `coprimitive subspace' $F_{k-g} \Lambda^k(\mathbb Z^{2g})$. Our main result in this direction identifies the subquotients of this filtration.

\begin{theorem}\label{thm:Lefschetz-intro}
The subquotients of the Lefschetz filtration \[\textup{gr}_r \Lambda^{k+2r}(\mathbb Z^{2g}) = F_r \Lambda^{k+2r}(\mathbb Z^{2g})/F_{r-1} \Lambda^{k+2r}(\mathbb Z^{2g})\] are free abelian groups isomorphic as $\textup{Sp}(2g)$-modules to $P^k(\mathbb Z^{2g}).$ 
\end{theorem}

The isomorphism is very explicit, given by contraction against $\omega^r/r!$, and permits explicit computation of associated graded maps; see Corollary \ref{cor:Lefschetz-graded}. We use this to give three applications. 

 We compute the cokernel of the map \[\omega^i/i!: \Lambda^{g-i}(\mathbb Z^{2g}) \to \Lambda^{g+i}(\mathbb Z^{2g})\] as an abelian group, and as a $\text{Sp}(2g)$-module \textit{up to filtration}.

\begin{theorem}\label{thm:hardLefschetz-intro}
The cokernel of $\omega^i/i!: \Lambda^{g-i}(\mathbb Z^{2g}) \to \Lambda^{g+i}(\mathbb Z^{2g})$ admits a $\textup{Sp}(2g)$-invariant $\mathbb Z$-split filtration whose associated graded modules are \[\textup{gr}_r \coker(\omega^i/i!) = P^{g-i-2r}(\mathbb Z^{2g})\big/\scalebox{1.2}{$\binom{r+i}{i}$}.\]
\end{theorem}

It would be interesting to determine this cokernel for arbitrary K\"ahler manifolds with integral K\"ahler form.\\

Studying the cokernel of $\wedge \omega$, our attention was drawn to the results of \cite{LeePacker}. These compute the group cohomology of the integer Heisenberg group $N_g$, a nilpotent central extension of $\mathbb Z^{2g}$ by $\mathbb Z$, and the result is a direct sum of groups of shape $(\mathbb Z/k)^{\binom{2g}{j} - \binom{2g}{j-2}}$. One is immediately drawn to the exponent, which is the dimension of the primitive subspace $P^j(\mathbb Z^{2g})$; this interpretation does not appear in \cite{LeePacker}, whose arguments were combinatorial. Our Lefschetz decomposition gives us an alternate proof of their results which makes the description in terms of primitive subspaces more transparent and respects the action of the symplectic group.

\begin{theorem}\label{thm:Heisenberg-intro}
If $N_g$ is the integer Heisenberg group of rank $2g+1$, then the group homology $H_k(N_g)$ admits a $\textup{Sp}(2g)$-invariant and $\mathbb Z$-split filtrations whose subquotients, as $\textup{Sp}(2g)$-modules, are isomorphic to $P^i(\mathbb Z^{2g})/(j)$ for various $i, j$.
\end{theorem}

See Theorem \ref{thm:LP-Lefschetz} for a more precise statement.\\

Our original interest in the results of \cite{LeePacker} and a Lefschetz decomposition over $\mathbb Z$ arose from a problem in Floer homology of $3$-manifolds. We omit certain technical details in this introduction, and in particular we suppress the dependence on spin$^c$ structures; the spin$^c$ structure of interest to us is always torsion. Section \ref{sec:HF-bg} contains a more detailed review.

The Heegaard Floer homology groups $HF^\bullet(Y)$ are invariants of closed, oriented $3$-manifolds introduced by Ozsvath and Szabo in \cite{OzSz3D1}, which are relatively graded $R[U]$-modules with $|U| = -2$. These invariants have had profound applications to the topology of $3$- and $4$-manifolds; see \cite{HFSurvey1,HFSurvey2,HFSurvey3,HFSurvey4} for several surveys of the subject.

Understanding the simplest of these invariants, $HF^\infty(Y; R)$, is a basic starting point for understanding the whole package. While the other Heegaard Floer groups depend on delicate geometric information about $Y$, the group $HF^\infty(Y;R)$ depends only on the group $H^1(Y; \mathbb Z)$ and the triple cup product $3$-form $\cup^3_Y: \Lambda^3 H^1(Y; \mathbb Z) \to \mathbb Z$. Its isomorphism type over $R = \mathbb Q$ and $R = \mathbb F_2$ are known, thanks to \cite[Proposition 35.1.5]{KMBook} and \cite{Lidman}, respectively. Despite its relative simplicity, the determination of the isomorphism type of $HF^\infty(Y;\mathbb Z)$ has largely remained open since this invariant was introduced.

There is another invariant with the same dependence on $\cup^3_Y$, named and studied in \cite{Mark}. The triple cup product defines a $3$-form $\omega_Y \in \Lambda^3 H^1(Y;\mathbb Z)^*$, and the \emph{cup homology} $HC_*(Y;R)$ is the homology of the complex $\Lambda^*(H^1(Y;\mathbb Z)) \otimes_{\mathbb Z} R[U, U^{-1}]$ with respect to the differential given by contraction with the triple cup product $\omega_Y$. 

Ozsv\'ath and Szab\'o found a spectral sequence of $R[U]$-modules \[\Lambda^*(H^1(Y;\mathbb Z)) \otimes R \Rightarrow HF^\infty(Y; R)\] and made the following conjecture \cite[Conjecture 4.10]{OzSzPlumbed}.

\begin{conj}\label{conj:OS}
If $Y$ is a closed oriented $3$-manifold, the $E^4$ page of the spectral sequence above is isomorphic to $HC_*(Y;R)$ as an $R[U]$-module, and the spectral sequence collapses at the $E^4$ page.
\end{conj}

This means, in particular, that $HF^\infty(Y; R)$ carries a filtration by $R[U]$-modules whose associated graded module is isomorphic to $HC_*(Y; R)$. When $R$ is a field, it implies these modules are isomorphic; when $R = \mathbb Z$, it is entirely plausible that there exists a $3$-manifold with $b_1(Y) = 3$ and \[HF^\infty(Y; \mathbb Z) \cong (\mathbb Z^6 \oplus \mathbb Z/4)[U, U^{-1}], \quad \quad HC_*(Y; \mathbb Z) \cong (\mathbb Z^6 \oplus \mathbb Z/2 \oplus \mathbb Z/2)[U, U^{-1}].\]

In \cite{LME1}, Francesco Lin and the second author found an algorithmic way to compute $HF^\infty(Y; R)$ as an $R[U]$-module, passing through the perspective of monopole Floer homology. We expected to use this algorithm to find either an example where the spectral sequence above fails to degenerate, or where there are `extension problems', so that $HC_*$ and $HF^\infty$ fail to be isomorphic as $R[U]$-modules. To our surprise, millions of computer calculations gave us isomorphic results, suggesting the following stronger conjecture:

\begin{conj}\label{conj:iso}
If $Y$ is a closed oriented $3$-manifold, then $HF^\infty(Y; \mathbb Z) \cong HC_*(Y;\mathbb Z)$. 
\end{conj}

The fact that these groups appear to be isomorphic --- without extension problems --- suggests that they should be isomorphic for a better reason than the collapsing of a spectral sequence. 

As discussed in Section \ref{sec:HF-bg}, $HC_*$ also extends to a functor on an appropriate cobordism category. One interpretation of the idea that these groups should be isomorphic for `a good reason' is as follows.

\begin{quest}\label{q:natural}
Do the assignments \[Y \mapsto HC_*(Y;R), \quad \quad Y \mapsto HF^\infty(Y; R)\] define naturally isomorphic functors on an appropriate cobordism category?
\end{quest}

Our third application of the Lefschetz filtration is a negative answer to this question. 

\begin{theorem}\label{thm:nonisofunctor}
The functors \[HC_*(-;\mathbb F_2), HF^\infty(-;\mathbb F_2): \mathsf C \to \mathbb F_2[U]\mathsf{-Mod}\] are not naturally isomorphic. 
\end{theorem}

\begin{remark}
The functoriality of $HF^\infty$ over the integers has not yet been completely worked out in the literature. Nevertheless, if one assumes that $HF^\infty(-;\mathbb Z)$ is indeed functorial --- in the sense of Postulate \ref{post:HF-works} below --- the same result holds over $\mathbb Z$.
\end{remark}

The basic idea proceeds as follows. First, given a functor $F: \mathsf C \to R[U]{-}\mathsf{Mod}$, the module $F(Y)$ inherits a canonical action of the oriented mapping class group $\text{MCG}^+(Y)$. We proceed to compute as much of this action as possible in the case $Y = \Sigma_g \times S^1$. As a module over the mapping class group, the invariant $HC_*(Y)$ is closely related to the kernel and cokernel of contraction with the canonical symplectic 2-form $\omega \in \Lambda^2(\mathbb Z^{2g})$, while the invariant $HF^\infty(Y)$ is instead related to the kernel and cokernel of contraction with the inhomogeneous form \[e^{\omega U - 1} = \omega U + \frac{\omega^2}{2} U^2 + \frac{\omega^3}{6} U^3 + \cdots \in \Lambda^*(\mathbb Z^{2g})[U, U^{-1}].\]

Taking advantage of Theorem \ref{thm:Lefschetz-intro}, we are largely able to compute $HF^\infty(Y; \mathbb Z)$ and $HC_*(Y; \mathbb Z)$ as modules over an index two subgroup $\text{MCG}^{++}(\Sigma_g \times S^1)$ of the oriented mapping class group in the case $Y = \Sigma_g \times S^1$, given precisely in Definition \ref{def:MCG-pp}. The sum total of our calculations is the following result. 

\begin{theorem}\label{thm:HFHC-filt-intro}
Let $Y = \Sigma_g \times S^1$. Then $HF^\infty(Y)$ and $HC_*(Y)$ compare as follows:

\begin{enumerate}[label=(\alph*)]
\item For all $g$, the $\mathbb Z[U]$-modules $HF^\infty(Y;\mathbb Z)$ and $HC_*(Y;\mathbb Z)$ are isomorphic. 
\item Suppose $g \ge 3$ and that Heegaard Floer homology is functorial over $\mathbb Z$. Then there exist $\textup{MCG}^{++}(Y)$-invariant filtrations of $HF^\infty$ and $HC_*$ so that we have, as modules over the mapping class group, \[\textup{gr}_k HF^\infty(Y; \mathbb Z) \cong \textup{gr}_k HC_*(Y;\mathbb Z) \cong \begin{cases} \bigoplus_{i=0}^k P^{g-k}(\mathbb Z^{2g})/(i) & 0 \le k \le g \\ \bigoplus_{j=0}^g P^j(\mathbb Z^{2g}) & k = g+1 \\ 
0 & \text{otherwise.} \end{cases}\]
\end{enumerate}
\end{theorem}

The first statement confirms both Conjecture \ref{conj:OS} and the stronger Conjecture \ref{conj:iso} for $Y = \Sigma_g \times S^1$. The situation in the second statement is worse over $\mathbb F_2$; see Remark \ref{rmk:F2-case}. While the second statement appears to give positive evidence towards Question \ref{q:natural}, the same computations falsify it, proving Theorem \ref{thm:nonisofunctor}:  

\begin{theorem}\label{thm:no-iso-intro}
For $g = 4$, the $\mathbb F_2[U]$-modules $HF^\infty(Y;\mathbb F_2)$ and $HC_*(Y;\mathbb F_2)$ are not \textit{equivariantly} isomorphic with respect to the action of $\textup{MCG}^{++}(Y)$.
\end{theorem}

\noindent This relies on a small computer calculation, isolated in Lemma \ref{lemma:computer-work}.

Combining Theorem \ref{thm:HFHC-filt-intro}(a) with Lee--Packer's computation Theorem \ref{thm:LP} and summing over all $k$, we also obtain an integral computation of the Heegaard Floer groups and find $p$-torsion for all $g \ge 2p-1$, generalizing the result of \cite[Corollary 4.11]{JabukaMark}:

\begin{cor}\label{cor:Z-formula}
There is an isomorphism of $\mathbb Z[U]$-modules
\[HF^\infty(\Sigma_g \times S^1; \mathbb Z) \cong \left(\mathbb Z^{d(g)} \oplus \bigoplus_{n \ge 2} (\mathbb Z/n)^{d_n(g)}\right)[U, U^{-1}],\]
where \[d(g) = 2\binom{2g}{g}, \quad d_n(g) = 2 \binom{2g+1}{g+1-2n}.\] In particular, for $p$ prime, $HF^\infty(\Sigma_g \times S^1)$ contains $p^k$-torsion if and only if $g\ge 2p^k-1$.  
\end{cor}

If one is only interested in this result, by \cite[Remark 4.9]{JabukaMark} it suffices to identify a certain pair of cokernels, which we do in Proposition \ref{prop:computation-for-coker}(a).

\begin{remark}
The equivariant isomorphism of \ref{thm:HFHC-filt-intro}(b) suggests that Question \ref{q:natural} is just barely false. The strongest form of Conjecture \ref{conj:OS} which is consistent with our results is as follows: there exists a natural $\mathbb Z$-split filtration on $HF^\infty(Y;\mathbb Z)$ whose associated graded functor is $HC_*(Y;\mathbb Z)$. It seems that the functor $HF^\infty(Y)$ contains `higher terms' which are not accessible to the cohomology ring.
\end{remark}

\begin{remark}
Monopole Floer homology also gives a functor $\overline{HM}(Y;R)$ whose underlying graded $R[U]$-module is known to be isomorphic to $HF^\infty(Y;R)$ \cite{HFvsHM-1-Team1,HFvsHM-1-Team2}; they are expected but not known to be isomorphic as functors, and share many formal properties. Our argument most likely gives that as functors we have $\overline{HM} \not\cong HC_*$ over $\mathbb Z$ or $\mathbb F_2$, but we use certain computational facts which have been established in the literature for $HF^\infty$ and not for $\overline{HM}$. Francesco Lin and the second author showed in \cite{LME2} that $\overline{HM}(Y)$ admits a filtration for which the associated graded map of $\overline{HM}(W)$ is equal to $HC_*(W)$; the higher terms appear to relate to the geometry of Dirac operators associated with $Y$. This is consistent with the discussion in the previous remark. 
\end{remark}

\subsection*{Conventions}
If $\Sigma_g = \#^g T^2$ is an oriented surface, we write $\omega$ for the cup product $2$-form on $H^1(\Sigma_g; \mathbb Z)$. We will use the same symbol to denote the expression of $\omega$ with respect to the standard symplectic basis $\mathbb Z^{2g} \cong H^1(\Sigma_g; \mathbb Z)$, in which \[\omega = e^1 \wedge e^2 + \cdots +e^{2g-1} \wedge e^{2g};\] the value of $g$ will usually be clear from context. The $r$th power of $\omega$ is divisible by $r!$, and we denote \[\omega_r = \frac{\omega^r}{r!}.\]

\subsection*{Organization}
In Section \ref{sec:Lef} we prove Theorem \ref{thm:Lefschetz-intro} and discuss properties of the Lefschetz filtration. In Section \ref{sec:applications} we apply these calculations to prove Theorems \ref{thm:hardLefschetz-intro} and \ref{thm:Heisenberg-intro}, as well as giving the algebraic input to Theorem \ref{thm:HFHC-filt-intro}. In Section \ref{sec:HF-bg} we recall the relevant facts about Heegaard Floer homology and cup homology. In Section \ref{sec:compute-MCG} we prove Theorems \ref{thm:HFHC-filt-intro} and \ref{thm:no-iso-intro}.

\subsection*{Acknowledgements}
The first author thanks Hailey Gamer for her assistance with code for computations in an earlier version of this paper. The second author thanks Tye Lidman for a discussion on Dehn surgery, Ian Zemke for discussions about functoriality and signs in Heegaard Floer homology, and Allen Hatcher for an exchange on the history of Lemma \ref{lemma:MCG-calc}. The authors thank Vidhu Adhihetty, Tye Lidman, Ciprian Manolescu for comments on a draft of this article.

\section{The Lefschetz filtration}\label{sec:Lef}
\subsection{Linear algebra preliminaries} 
We make use of two products on $\Lambda^*(\mathbb Z^{2g})$. The first is the wedge product. The second is the \textit{interior product}, defined using the symplectic form; we follow \cite[Section 3.1]{JabukaMark}. Given $v \in \mathbb Z^{2g}$, we may define a map $\iota_v: \Lambda^k(\mathbb Z^{2g}) \to \Lambda^{k-1}(\mathbb Z^{2g})$ by the formula \[\iota_v(x_1 \wedge \cdots \wedge x_k) = \sum_{i=1}^k (-1)^{i-1} \omega(x_i, v) x_1 \wedge \cdots \wedge \hat x_i \wedge \cdots \wedge x_k.\] This obeys the signed Leibniz rule \[\iota_v(x \wedge y) = \iota_v(x) \wedge y + (-1)^{|x|} x \wedge \iota_v(y).\] Identifying $\mathbb Z^{2g} = \Lambda^1(\mathbb Z^{2g})$, there is then a unique extension to a map \[\iota: \Lambda^i(\mathbb Z^{2g}) \otimes \Lambda^j(\mathbb Z^{2g}) \to \Lambda^{j-i}(\mathbb Z^{2g})\] for which $\iota_{x \wedge y}(z) = \iota_x(\iota_y(z)),$ but it is important to note that $\iota_x$ usually does not satisfy the Leibniz rule when $|x| > 1$. 

The most important case is contraction by $\omega$, which satisfies a `weighted' Leibniz rule.

\begin{lemma}[{\cite[Lemma 3.1]{JabukaMark}}]\label{lemma:Lefschetz-identity}
For any $x \in \Lambda^k(\mathbb Z^{2g})$, we have \[\iota_\omega(\omega \wedge x) = \omega \wedge \iota_\omega(x) + (k-g)x.\]
\end{lemma}

This generalizes to the following formula, where here we interpret $\omega_0 = 1$ and $\omega_i = 0$ for $i < 0$. 

\begin{lemma}\label{lemma:general-mn-formula}
For any $x \in \Lambda^k(\mathbb Z^{2g})$, we have \[\iota_{\omega_m}(\omega_n \wedge x) = \sum_{j=0}^m 
(-1)^j \binom{g-k+m-n}{j} \omega_{n-j} \wedge \iota_{\omega_{m-j}}(x).\]
\end{lemma}

\begin{remark}
    Here we take the convention that $\binom{-i}{j} = (-1)^j \binom{i+j-1}{j}$ for all integers $i$. 
\end{remark}

\begin{proof}
The statement is proved by induction first on $n$, then on $m$, with base case $(m,n) = (1,1)$ established by Lemma \ref{lemma:Lefschetz-identity}. The statement for $m = 1$ is \[\iota_\omega(\omega_n \wedge x) = \omega_n \wedge \iota_\omega(x) + (k-g+n-1) \omega_{n-1} \wedge x,\] and the inductive step is
\begin{align*}
\iota_\omega(\omega_{n+1} \wedge x) &= \frac{1}{n+1} \iota_\omega\left(\omega \wedge (\omega_n \wedge x)\right) \\
&= \frac{1}{n+1}\left(\omega \wedge \iota_\omega(\omega_n \wedge x) + (k + 2n - g) \omega_n \wedge x\right) \\
&= \frac{1}{n+1} \left(\omega \wedge (\omega_n \wedge \iota_\omega(x) + (k-g+n-1) \omega_{n-1} \wedge x) + (k+2n-g) \omega_n \wedge x\right) \\
&= \omega_{n+1} \wedge \iota_\omega(x) + (k-g+(n+1)-1) \omega_n \wedge x.
\end{align*}
The inductive step when inducting on $m$ is 
\begin{align*}
    \iota_{\omega_{m+1}}(\omega_n \wedge x) = \frac{1}{m+1} \iota_\omega\left(\iota_{\omega_m}(\omega_n \wedge x)\right) 
    &= \frac{1}{m+1} \iota_\omega\left(\sum_{j=0}^m (-1)^j \binom{g-k+m-n}{j} \omega_{n-j} \wedge \iota_{\omega_{m-j}}(x)\right) \\
    &= \frac{1}{m+1} \sum_{j=0}^m (-1)^j \binom{g-k+m-n}{j} \iota_\omega\left(\omega_{n-j} \wedge \iota_{\omega_{m-j}}(x)\right) 
\end{align*}
To prove the desired relation, expand each summand with the inductive hypothesis: \begin{align*}
\iota_\omega(\omega_{n-j} \wedge \iota_{\omega_{m-j}}(x)) 
= (m-j+1) \omega_{n-j} \wedge \iota_{\omega_{m-j+1}}(x) + (k-g+n-2m+j-1) \omega_{n-j-1} \wedge \iota_{\omega_{m-j}}(x).
\end{align*}
Each term $\omega_{n-j} \wedge \iota_{\omega_{m+1-j}}$ for $j \in \{0,m+1\}$ appears once with the correct coefficient, while $0 < j \le m$ appears twice. Combine them with the relation \begin{align*} (m+1) &\binom{g-k+m+1-n}{j} \\
&= (m-j+1) \binom{g-k+m-n}{j} - (k-g+n-2m+j-2) \binom{g-k+m-n}{j-1}.\qedhere\end{align*} 
\end{proof}

Following \cite[Section 3.1]{JabukaMark}, contraction also allows us to define the $\text{Sp}(2g)$-equivariant Hodge--Lefschetz duality operator $*: \Lambda^{g-k}(\mathbb Z^{2g}) \to \Lambda^{g+k}(\mathbb Z^{2g})$ by the formula $*x = \iota_x(\omega_g)$. The Hodge--Lefschetz star is an isomorphism by \cite[Proposition 3.3]{JabukaMark}, and in fact $*^2$ acts as $(-1)^k$ on $\Lambda^{g-k}$. Most importantly for our purposes, duality interchanges the role $\wedge \omega$ and $\iota_\omega$: we have \[*(\omega \wedge x) = \iota_{\omega \wedge x}(\omega_g) = \iota_\omega \iota_x(\omega_g) = \iota_\omega(*x).\]

\subsection{A Lefschetz filtration}
Instead of attempting to define a decomposition directly in terms of primitive subspaces, we study a natural extension of the notion of `primitive subspace'.

\begin{definition}
The \textbf{Lefschetz filtration} on the exterior powers is defined by the formula 
\[F_r \Lambda^k(\mathbb Z^{2g}) := \{\alpha \in \Lambda^k(\mathbb Z^{2g}) \mid \omega^{g-k+1+r} \wedge \alpha = 0\}.\]
\end{definition}

Notice that because the action of $\text{Sp}(2g)$ on $\Lambda^*$ commutes with wedge products and fixes $\omega$, the subspaces $F_r \Lambda^k(\mathbb Z^{2g})$ are $\text{Sp}(2g)$-invariant. 

It is straightforward to see that this filtration splits over $\mathbb Z$ --- the subquotients are free abelian groups, which we will study in detail shortly --- but it rarely splits over $\mathbb Z[\text{Sp}(2g)]$. For example, it is not hard to check that the short exact sequence \[0 \to F_0 = P^2(\mathbb Z^{2g}) \to \Lambda^2(\mathbb Z^{2g}) = F_1 \to F_1/F_0 \cong \mathbb Z \to 0\] admits no $\text{Sp}(2g)$-equivariant splitting for any $g \ge 2$; see Lemma \ref{lemma:no-equivariant-splitting}. \\

For $0 \le k \le g$, the first nonzero term in the filtration is $F_0 \Lambda^k(\mathbb Z^{2g}) = P^k(\mathbb Z^{2g})$, while for $g \le k \le 2g$, the first nonzero term in the filtration is \[F_{k-g} \Lambda^k(\mathbb Z^{2g}) = \{\alpha \in \Lambda^k(\mathbb Z^{2g}) \mid \omega \wedge \alpha = 0\},\] sometimes called the \textit{coprimitive subspace} $\tilde P^k$. 

First, we will discuss the behavior of our filtration with respect to the wedge and contraction operations. 

\begin{lemma}\label{lemma:contr-filt}
For all $x \in \Lambda^k(\mathbb Z^{2g})$, we have 
\[\omega_j \wedge x \in F_{r+j} \Lambda^{k+2j} \iff x \in F_r \Lambda^k \iff \iota_{\omega_j}(x) \in F_{r-j} \Lambda^{k-2j}.\]
\end{lemma}
\begin{proof}
It suffices to prove this for $j = 1$. Explicitly, we aim to show 
\[\omega^{g-k+r} \wedge \omega x = 0 \iff \omega^{g-k+r+1} \wedge x = 0 \iff \omega^{g-k+r+2} \iota_\omega(x) = 0.\] The first biconditional is a tautology. For the second, use the formula \[\iota_\omega(\omega_{g-k+r+j+1} \wedge x) = \omega_{g-k+r+j+1} \wedge \iota_\omega(x) + (r+j) \omega_{g-k+r+j} \wedge x.\] Taking $j = 1$ proves the forward implication. For the reverse implication, inducting downwards from $j = k-r$ implies $\omega_{g-k+r+j+1} \wedge x = 0$ for all $j \ge 0$.
\end{proof}

This implies that duality respects the Lefschetz filtration:

\begin{lemma}
    The Hodge--Lefschetz map restricts to an isomorphism $*: F_r \Lambda^{g-k}(\mathbb Z^{2g}) \to F_{k+r} \Lambda^{g+k}(\mathbb Z^{2g}).$
\end{lemma}
\begin{proof}
  Take $x \in F_r\Lambda^{g-k}(\mathbb Z^{2g})).$ By definition, $\omega^{k+r+1} \wedge x=0,$ which implies that $$\iota_{\omega^{k+r+1}}(*x) =*(\omega^{k+r+1} \wedge x)=0 \in F_{-1} \Lambda^{g-k-2r-2}(\mathbb Z^{2g}).$$ By Lemma \ref{lemma:contr-filt}, we have $*x \in F_{k+r}\Lambda^{g+k}(\mathbb Z^{2g}).$  
\end{proof}

We also obtain injectivity of contraction on associated graded spaces:

\begin{lemma}
For any $0 \le r \le g-k$ and $0 \le k \le g$, contraction against $\omega_r$ gives a well-defined injective homomorphism \[\iota_{\omega_r}: \textup{gr}_{r+j} \Lambda^{k+2r}(\mathbb Z^{2g}) \to P^j(\mathbb Z^{2g}).\]
\end{lemma}
\begin{proof}
Well-definedness and injectivity of this map are the two directions in the biconditional
\[x \in F_{r-1} \Lambda^{k+2r}(\mathbb Z^{2g}) \iff \iota_{\omega_r}(x) \in F_{-1} \Lambda^k(\mathbb Z^{2g}) = 0.\qedhere\]
\end{proof}

\noindent What is special to the linear case is that this map is also surjective. The following statement will fail for $H^*(X;\mathbb Z)$ with $X$ smooth projective; $\iota_{\omega_r}$ will not be an isomorphism for $k = 0, r = g$ unless $X = \mathbb{CP}^d$. 

\begin{prop}\label{prop:w-surj}
For all integers $g$, for all $0 \le k \le g$ and all $0 \le r \le g-k$, the contraction map \[\iota_{\omega_r}: F_r \Lambda^{k+2r}(\mathbb Z^{2g}) \to P^k(\mathbb Z^{2g})\] is surjective. 
\end{prop}

\begin{proof}
We will prove this claim by induction on $g$, where the base case $g=0$ is tautological. Suppose the statement is known for forms in $\mathbb Z^{2g}$; we will prove the statement for forms in $\mathbb Z^{2g+2}$. The case $k = 0$ follows from the observation that $\omega_r \in F_r \Lambda^{2r}$ and $\iota_{\omega_r}(\omega_r) = (-1)^r$. The case $r = g-k$ follows from Lemma \ref{lemma:general-mn-formula}: if $x$ is primitive, then $\iota_{\omega_j}(x) = 0$ for all $j > 0$, and then \[\iota_{\omega_{g-k}}(\omega_{g-k} x) = (-1)^{g-k} \binom{g-k}{g-k} x = (-1)^{g-k} x,\] where $\omega_{g-k} x \in F_{g-k} \Lambda^{2g-k}$.  

Now fix $r$. Consider $\mathbb Z^{2g}$ as the span of the first $2g$ coordinates in $\mathbb Z^{2g+2}$. Write $\omega$ for the standard $2$-form on $\mathbb Z^{2g+2}$ and $\eta$ for the standard $2$-form on $\mathbb Z^{2g}$, so we have for $r \ge 1$ that 
\[\omega_r = \eta_r + e^{2g+1} e^{2g+2} \eta_{r-1}.\] 
Consider $x \in P^{k+1}(\mathbb Z^{2g+2})$ for $0 \le k < g-r$, which means that $\omega_{g-k+1} x = 0$. Write 
\[x = a + e^{2g+1} b + e^{2g+2} c + e^{2g+1} e^{2g+2} d,\] 
where $a,b,c,d \in \Lambda^*(\mathbb Z^{2g}).$ We may rewrite $\omega_{g-k+1} x = 0$ as 
\[0 = \omega_{g-k+1} x = \eta_{g-k+1} a + e^{2g+1} \eta_{g-k+1} b + e^{2g+2} \eta_{g-k+1} c + e^{2g+1} e^{2g+2} (\eta_{g-k} a + \eta_{g-k+1} d).\]
Because the individual summands must be zero, we obtain 
\[b, c \in P^k(\mathbb Z^{2g}), \quad a \in F_1 \Lambda^{k+1}(\mathbb Z^{2g}), \quad d \in P^{k-1}(\mathbb Z^{2g}).\] 
By inductive hypothesis and the fact that $r \le g-k-1$ we have 
\[b = \iota_{\eta_r}(b'), \quad c = \iota_{\eta_r}(c'), \quad d = \iota_{\eta_r}(d') = \iota_{\eta_{r+1}}(d''),\] 
where 
\[b', c' \in F_r \Lambda^{k+2r}(\mathbb Z^{2g}), \quad d' \in F_r \Lambda^{k+2r-1}(\mathbb Z^{2g}), \quad d'' \in F_{r+1} \Lambda^{k+2r+1}(\mathbb Z^{2g}).\] 
Set $y = d'' + e^{2g+1} b' + e^{2g+2} c' + e^{2g+1} e^{2g+2} d'.$ We have 
\[\iota_{\omega_{r+1}}(y) = \iota_{\eta_{r+1}}(d'') - \iota_{\eta_r}(d') + e^{2g+1} \iota_{\eta_{r+1}}(b') + e^{2g+2} \iota_{\eta_{r+1}}(c') + e^{2g+1} e^{2g+2} \iota_{\eta_{r+1}}(d') = d - d = 0.\] 
It follows from Lemma \ref{lemma:contr-filt} that $y \in F_r \Lambda^{k+2r+1}(\mathbb Z^{2g+2})$. Because $\iota_{\omega_r}(y) \in P^{k+1}(\mathbb Z^{2g+2})$ and 
\[\iota_{\omega_r}(y) = \iota_{\eta_r}(d'') - \iota_{\eta_{r-1}}(d') + e^{2g+1} b + e^{2g+2} c + e^{2g+1} e^{2g+2} d,\] 
we see that $x = \iota_{\omega_r}(y) + a'$ with $a' \in P^{k+1}(\mathbb Z^{2g+2}) \cap \Lambda^{k+1}(\mathbb Z^{2g})$. We conclude by showing $a' = 0$. Because
\[0 = \omega_{g-k+1} a' = \eta_{g-k+1} a' + e^{2g+1} e^{2g+2} \eta_{g-k} a'\] 
we have $\eta_{g-k} a' = 0$; because the map $\eta_{g-k}: \Lambda^k(\mathbb Z^{2g}) \to \Lambda^{2g-k}(\mathbb Z^{2g})$ is injective we see $a' = 0$.
\end{proof}

Thus, we have established the following `integral Lefschetz decomposition'.

\begin{theorem}\label{thm:Lefschetz}
Contraction against $\omega_r$ defines a $\textup{Sp}(2g)$-equivariant isomorphism \[\iota_{\omega_r}: \textup{gr}_r \Lambda^k(\mathbb Z^{2g}) = \frac{F_r \Lambda^k(\mathbb Z^{2g})}{F_{r-1} \Lambda^k(\mathbb Z^{2g})} \cong P^{k-2r}(\mathbb Z^{2g}).\qedhere\]
\end{theorem}

Because the interaction between contraction and wedge product is well-behaved, we obtain the following result. 

\begin{cor}\label{cor:Lefschetz-graded}
With respect to the isomorphism of Theorem \ref{thm:Lefschetz}, for any $0 \le r \le g-k$, the maps given by wedge product and contraction with $\omega_j$ 
\begin{align*}
P^k(\mathbb Z^{2g}) \xrightarrow{\iota_{\omega_r}^{-1}} &\textup{gr}_r \Lambda^{k+2r}(\mathbb Z^{2g}) \xrightarrow{\wedge \omega_j} \textup{gr}_{r+j} \Lambda^{k+2r+2j}(\mathbb Z^{2g}) \xrightarrow{\iota_{\omega_{r+j}}} P^k(\mathbb Z^{2g}) \\
P^k(\mathbb Z^{2g}) \xrightarrow{\iota_{\omega_{r+j}}^{-1}} &\textup{gr}_{r+j} \Lambda^{k+2r+2j}(\mathbb Z^{2g}) \xrightarrow{\iota_{\omega_j}} \textup{gr}_r \Lambda^{k+2r}(\mathbb Z^{2g}) \xrightarrow{\iota_{\omega_r}} P^k(\mathbb Z^{2g})
\end{align*}
are multiplication by $(-1)^j \binom{g-k+r}{j}$ and $\binom{r+j}{j},$ respectively.
\end{cor}
\begin{proof}
The claim is that for $x \in \text{gr}_r \Lambda^{k+2r}(\mathbb Z^{2g})$ and $y \in \text{gr}_{r+j} \Lambda^{k+2r+2j}(\mathbb Z^{2g})$, we have \[\iota_{\omega_{r+j}}(\omega_j x) = (-1)^j \binom{g-k+r}{j} \iota_{\omega_r}(x), \quad \quad \iota_{\omega_r}(\iota_{\omega_j}(y)) = \binom{r+j}{j} \iota_{\omega_{r+j}}(y).\] The latter claim follows from $\iota_{ab}(y) = \iota_a(\iota_b(y))$ and the fact that $\omega_j \cdot \omega_r = \binom{r+j}{j} \omega_{r+j}$. The former claim follows from this by Hodge--Lefschetz duality. For a more direct proof, Lemma \ref{lemma:general-mn-formula} gives \[\iota_{\omega_{r+j}}(\omega_j \wedge x) = \sum_{i=0}^{r+j} (-1)^i \binom{g-k+r}{i} \omega_{j-i} \wedge \iota_{\omega_{r+j-i}}(x).\] Now $\iota_{\omega_{r+j-i}}(x) \in F_{i-j}$ and so can only be nonzero if $i \ge j$, whereas the term $\omega_{j-i}$ can only be nonzero if $i \le j$. Thus the only nonzero term appears for $i=j$, where it gives $(-1)^j \binom{g-k+r}{j} \iota_{\omega_r}(x)$.
\end{proof}

In particular, the isomorphism $\iota_{\omega_r}^{-1}: P^k(\mathbb Z^{2g}) \to \text{gr}_r \Lambda^{k+2r}(\mathbb Z^{2g})$ is given explicitly by \[\iota_{\omega_r}^{-1}(x) = \frac{(-1)^r}{\binom{g-k}{r}} \left[\omega_r \wedge x\right],\] and in particular the element $[\omega_r \wedge x] \in \text{gr}_r \Lambda^{k+2r}(\mathbb Z^{2g})$ is always uniquely divisible by $\binom{g-k}{r}$. 

\subsection{Splittings of the Lefschetz filtration}
Now that we have a filtration with associated graded pieces isomorphic to primitive subspaces, we discuss the extent to which this allows for a true direct sum decomposition in terms of these subspaces.

\begin{definition}
Suppose $A$ is a $\mathbb Z[G]$-module, and $0 \subset F_0 A \subset \cdots \subset F_n A = A$ is a filtration by $\mathbb Z[G]$-submodules. 

A $\mathbb Z$-\textbf{splitting} of this filtration is a splitting of each sequence $0 \to F_{r-1} A \to F_r A \to \textup{gr}_r A \to 0$; equivalently, a sequence of subgroups $G_r A \subset F_r A$ for $0 \le r \le n$ for which \[G_r A \cap F_{r-1} A = 0, \quad F_r A = G_r A + F_{r-1} A.\] The splitting is said to be $G$-\textbf{equivariant} if the splitting is one of $\mathbb Z[G]$-modules; equivalently, each $G_r A$ is a $\mathbb Z[G]$-submodule of $F_r A$. A filtration which admits a $\mathbb Z$-splitting is said to be $\mathbb Z$-\textbf{split}.
\end{definition}

Given a $\mathbb Z$-splitting of a filtration, the group $A$ enjoys a direct sum decomposition $A = \oplus_{i=0}^n G_i A$ for which the filtration is given by $F_r A = \oplus_{i=0}^r G_i A$.\\

We will now discuss the extent to which the Lefschetz filtration admits splittings with certain favorable properties. Equivariant splittings are out of the question:

\begin{lemma}\label{lemma:no-equivariant-splitting}
    The Lefschetz filtration on $\Lambda^{2k}(\mathbb Z^{2g})$ does not admit a $\textup{Sp}(2g)$-equivariant splitting for any $g \ge 2$ and $0 < k < g$. 
\end{lemma}
\begin{proof}
Because $F_{k-1} \Lambda^{2k}(\mathbb Z^{2g})$ has quotient $\mathbb Z$ with trivial $\text{Sp}(2g)$-action, we seek a $\text{Sp}(2g)$-invariant element $x \in \Lambda^{2k}(\mathbb Z^{2g})$ for which $\Lambda^{2k} = F_{k-1} + \langle x \rangle$. 

Over the complex numbers, the Lefschetz filtration is equivariantly split, with $\mathbb C \langle \omega_k\rangle = \Lambda^{2k}(\mathbb C^{2g})^{\text{Sp}(2g)}$. It follows that if such an integer form $x$ exists, it is a complex multiple of $\omega_k$. Because $\omega_k$ is not divisible over the integers, we must have $x = \pm\omega_k$. However, by Corollary \ref{cor:Lefschetz-graded}, the form $\omega_k$ represents $\binom{g}{k}$ times a generator in $\text{gr}_k \Lambda^{2k} \cong \mathbb Z$, and hence $F_{k-1} + \langle \omega_k\rangle$ is not the whole of $\Lambda^{2k}$ for $g \ge 2$ and $0 < k < g$.
\end{proof}

We will be using the Lefschetz filtration to understand the action of $\wedge \omega_k$ and $\iota_{\omega_k}$ on the exterior algebra. Because the subquotients are free abelian groups, the Lefschetz filtration on each $\Lambda^k(\mathbb Z^{2g})$ is $\mathbb Z$-split. It would be ideal if these could all be chosen to be compatible with the action of $\iota_\omega$, but this is not so:

\begin{lemma}\label{lemma:no-total-filtration}
For $g \ge 2$, there is no $\mathbb Z$-splitting of the Lefschetz filtration on $\Lambda^{2g-2}(\mathbb Z^{2g})$ for which $\iota(G_g \Lambda^{2g}(\mathbb Z^{2g})) \subset G_{g-1} \Lambda^{2g-2}(\mathbb Z^{2g})$. 
\end{lemma}
\begin{proof}
The group $G_g \Lambda^{2g}(\mathbb Z^{2g})$ is generated by $\omega_g$. By Corollary \ref{cor:Lefschetz-graded}, the map \[\iota_\omega: \mathbb Z \cong \text{gr}_g \Lambda^{2g}(\mathbb Z^{2g}) \to \text{gr}_{g-1} \Lambda^{2g-2}(\mathbb Z^{2g}) \cong \mathbb Z\] is given by multiplication by $g$. If there existed such a splitting of the Lefschetz filtration on $\Lambda^{2g-2}(\mathbb Z^{2g})$, it would follow that $\iota_\omega(\omega_g) = gx$ for some integer form $x$. However, Lemma \ref{lemma:general-mn-formula} implies that $\iota_\omega(\omega_g) = -\omega_{g-1}$, which is not divisible by any integer larger than $1$.
\end{proof}

The issue only occurs in high degrees; splittings exist up until, roughly, the halfway point of the exterior algebra. 

\begin{prop}\label{prop:splitting-firsthalf}
For all $g$, there exist $\mathbb Z$-splittings of the Lefschetz filtration on $\Lambda^k(\mathbb Z^{2g})$ for each $0 \le k \le g$ so that for all such $k$, we have $\iota_\omega(G_r \Lambda^k) \subset G_{r-1} \Lambda^{k-2}$. 
\end{prop}
\begin{proof}
We prove this by induction on $k$, taking base cases $k = 0, 1$ where the filtrations are $1$-step. Pick $0 \le r \le k/2$ and consider the commutative diagram 
\[\begin{tikzcd}
	{G_{r-1} \Lambda^{k-2}} & {\iota_\omega^{-1}(G_{r-1} \Lambda^{k-2})} \\
	{P^{k-2r}}
	\arrow["{\iota_{\omega_{r-1}}}"', from=1-1, to=2-1]
	\arrow["{\iota_\omega}"', from=1-2, to=1-1]
	\arrow["{r \iota_{\omega_r}}", from=1-2, to=2-1]
\end{tikzcd}\]
The horizontal arrow has image $r G_{r-1} \Lambda^{k-2}$ and the vertical arrow is surjective, both by Corollary \ref{cor:Lefschetz-graded}. The diagonal arrow thus has image $r P^{k-2r}$, so that $\iota_{\omega_r}:\iota_\omega^{-1}(G_{r-1} \Lambda^{k-2}) \to P^{k-2r}$ is surjective. Let $G_r \Lambda^k$ be the image of any right inverse to this map.
\end{proof}

If one relaxes the compatibility demands between different degrees, one can still obtain a useful splitting of the filtration on the high-degree exterior algebras:

\begin{prop}\label{prop:splitting-across-midpoint}
For all $g$ and all $0 \le k < g$, if $\Lambda^{g-k}(\mathbb Z^{2g})$ is equipped with a $\mathbb Z$-splitting of its Lefschetz filtration, there exists a $\mathbb Z$-splitting of the Lefschetz filtration of $\Lambda^{g+k}(\mathbb Z^{2g})$ so that for all $r$, we have \[\iota_{\omega_k}\left(G_{r+k} \Lambda^{g+k}(\mathbb Z^{2g})\right) \subset G_r \Lambda^{g-k}(\mathbb Z^{2g}).\]
\end{prop}
\begin{proof}
The argument is similar, with the crucial diagram instead being 
\[\begin{tikzcd}
	{G_r \Lambda^{g-k}} & {\iota_{\omega_k}^{-1}(G_r \Lambda^k)} \\
	{P^{g-k-2r}} 
	\arrow["{\iota_{\omega_r}}"', from=1-1, to=2-1]
	\arrow["{\iota_{\omega_k}}"', from=1-2, to=1-1]
	\arrow["{\binom{r+k}{r} \iota_{\omega_{r+k}}}", from=1-2, to=2-1]
\end{tikzcd}\]
Corollary \ref{cor:Lefschetz-graded} says the vertical map is surjective and the horizontal map has image $\binom{r+k}{r} G_r \Lambda^{g-k}$, and one concludes in the same way.
\end{proof}

\section{Algebraic consequences of the Lefschetz filtration}\label{sec:applications}
\subsection{Failure of the Hard Lefschetz theorem for $\Lambda^*(\mathbb Z^{2g})$} Write \[C_{g,k} = \coker\left(\wedge \omega_k: \Lambda^{g-k}(\mathbb Z^{2g}) \to \Lambda^{g+k}(\mathbb Z^{2g})\right).\] 

\begin{theorem}\label{thm:failure-of-hard-lefschetz}
The Lefschetz filtration on $C_{g,k}$ is $\mathbb Z$-split, whose subquotients may be identified as $\textup{Sp}(2g)$-modules:
\[\textup{gr}_r C_{g,k} = \begin{cases} P^{g-k-2r}(\mathbb Z^{2g})/\binom{r+k}{r} & 0 \le r \le \lfloor \frac{g-k}{2} \rfloor \\ 0 & \textup{otherwise.} \end{cases}\]
\end{theorem}
\begin{proof}
By Hodge--Lefschetz duality, $C_{g,k}$ is isomorphic as an $\text{Sp}(2g)$-module to the cokernel of $\iota_{\omega_k}: \Lambda^{g+k} \to \Lambda^{g-k}$, so we work instead with contraction. Choosing a $\mathbb Z$-splitting of the Lefschetz filtration on $\Lambda^{g-k}$ arbitrarily and of $\Lambda^{g+k}$ following Proposition \ref{prop:splitting-across-midpoint}, the map $\iota_{\omega_k}$ is identified as the following direct sum, with Lefschetz filtration sent to the sum of the first $r$ summands: \[\bigoplus_{r=0}^{\lfloor \frac{g-k}{2}\rfloor} \left(\iota_{\omega_k}: G_{r+k} \Lambda^{g+k}(\mathbb Z^{2g}) \to G_r \Lambda^{g-k}(\mathbb Z^{2g})\right).\] 
In particular, for $0 \le r \le \lfloor \frac{g-k}{2} \rfloor$ we have \begin{align*}
F_r C_k &\cong \bigoplus_{i=0}^r \coker \left(\iota_{\omega_k}: G_{i+k} \Lambda^{g+k}(\mathbb Z^{2g}) \to G_i \Lambda^{g-k}(\mathbb Z^{2g})\right) \\
\text{gr}_r C_k &\cong \coker \left(\iota_{\omega_k}: G_{r+k} \Lambda^{g+k}(\mathbb Z^{2g}) \to G_r \Lambda^{g-k}(\mathbb Z^{2g})\right),
\end{align*}
the first isomorphism one of abelian groups, the latter of $\text{Sp}(2g)$-modules. To conclude, Corollary \ref{cor:Lefschetz-graded} identifies the associated graded maps with \[\binom{r+k}{r}: P^{g-k-2r}(\mathbb Z^{2g}) \to P^{g-k-2r}(\mathbb Z^{2g}). \qedhere\]
\end{proof}

\subsection{The $\text{Sp}(2g)$-action on the cohomology of Heisenberg groups}

The \textit{integer Heisenberg group} $N_g$ is the group of upper-triangular integer $(g+2) \times (g+2)$ matrices which are equal to $1$ on the diagonal, and whose nonzero off-diagonal entries lie in the first row or final column. A typical element $A \in N_2$ is given by \[A = \begin{pmatrix} 1 & x_1 & x_2 & z \\ 0 & 1 & 0 & y_1 \\ 0 & 0 & 1 & y_2 \\ 0 & 0 & 0 & 1 \end{pmatrix} = \begin{pmatrix} 1 & \vec x & z \\ 0_{2 \times 1} & I_{2 \times 2} & \vec y \\ 0 & \vec 0_{1 \times 2} & 1 \end{pmatrix} \in M_4(\mathbb Z).\] 

The cohomology groups $H^k(N_g;\mathbb Z)$ were calculated in \cite{LeePacker}. 

\begin{theorem}[{\cite[Theorem 1.8]{LeePacker}}]\label{thm:LP}
The cohomology of the integer Heisenberg group $N_g$ is, as an abelian group, 
\[H^k(N_g;\mathbb Z) = 
\begin{cases} \bigoplus_{j=0}^{\lfloor k/2 \rfloor} (\mathbb Z/j)^{\binom{2g}{k-2j} - \binom{2g}{k-2j-2}}, & 0 \le k \le g \\\\
\mathbb Z^{\binom{2g}{2g-k+1} - \binom{2g}{2g-k-1}} \oplus \left(\bigoplus_{j=1}^{\lfloor (2g-k+2)/2 \rfloor} (\mathbb Z/j)^{\binom{2g}{2g-k-2j+2} - \binom{2g}{2g-k-2j}}\right). & g+1 \le k \le 2g+1  \end{cases}
\]    
\end{theorem}

The proof is combinatorial in nature, and we will find some use for those techniques in the next section. Here we present an alternative proof which strengthens these results and demystifies the appearance of expressions like $\binom{2g}{k} - \binom{2g}{k-2}$. 

The group $N_g$ is a central extension of $\mathbb Z^{2g}$ by $\mathbb Z$, and this extension is classified by the element $\omega \in H^2(\mathbb Z^{2g};\mathbb Z) \cong \Lambda^2(\mathbb Z^{2g})$. It follows that the action of $\text{Sp}(2g)$ on $\mathbb Z^{2g}$ extends to an action on $N_g$, and thus that the cohomology groups of $H^k(N_g;\mathbb Z)$ are $\text{Sp}(2g)$-modules. 

We will prefer to state our result in terms of the homology of $N_g$. This is equivalent to working with cohomology because $BN_g$ is a closed orientable manifold of dimension $(2n+1)$, so by Poincar\'e duality we have $\text{Sp}(2g)$-equivariant isomorphisms $H^k(N_g;\mathbb Z) \cong H_{2g+1-k}(N_g;\mathbb Z)$. Because the argument and result is somewhat different for $k \le g+1$ and $k > g+1$, we split this into two statements. 

\begin{theorem}\label{thm:LP-Lefschetz}
The homology groups of the integer Heisenberg groups, considered as $\text{Sp}(2g)$-modules, admit $\mathbb Z$-split filtrations whose subquotients are primitive subspaces mod $j$ or their duals. More precisely, 
\begin{enumerate}[label=(\alph*)]
    \item For $0 \le k \le g$, the module $H_k(N_g;\mathbb Z)$ admits a $\mathbb Z$-split filtration with subquotients 
    \[\textup{gr}_r H_k(N_g;\mathbb Z) \cong \begin{cases} P^{k-2r-1}(\mathbb Z^{2g})/(r+1), & 0 \le r \le \lfloor \frac{k-1}{2} \rfloor \\ P^k(\mathbb Z^{2g}). & r = \lfloor \frac{k+1}{2} \rfloor \end{cases}\]
    
    \item For $g+1 \le k \le 2g+1$, write $\epsilon \in \{0,1\}$ for the parity of $k-1$. Then the module $H_k(N_g;\mathbb Z)$ admits a $\mathbb Z$-split filtration with subquotients     
    \[\textup{gr}_r H_k(N_g;\mathbb Z) \cong \begin{cases} 
    P^{2r+\epsilon}(\mathbb Z^{2g})\big/\big(\lfloor \frac{2g-k+1}{2}\rfloor - r\big), & 0 \le r \le \lfloor \frac{2g-k+1}{2} \rfloor \\ 
    \end{cases}\]
\end{enumerate}
\end{theorem}

\begin{proof} 
We first handle the case $k \le g$. Identifying $H_k(\mathbb Z^{2g};\mathbb Z) \cong \Lambda^k(\mathbb Z^{2g})$, the Gysin sequence in homology is an $\text{Sp}(2g)$-equivariant long exact sequence
\[\cdots \to H_k(N_g;\mathbb Z) \to \Lambda^k(\mathbb Z^{2g}) \xrightarrow{\iota_\omega} \Lambda^{k-2}(\mathbb Z^{2g}) \to H_{k-1}(N_g;\mathbb Z) \to \cdots \] In each degree $k$, this gives rise to an $\text{Sp}(2g)$-equivariant short exact sequence 
\[0 \to \coker(\iota_\omega: \Lambda^{k+1} \to \Lambda^{k-1}) \to H_k(N_g;\mathbb Z) \to \ker(\iota_\omega: \Lambda^k \to \Lambda^{k-2}) \to 0;\] we abbreviate the first term to $\coker(\omega)_{k-1}$ and the latter term to $\ker(\omega)_k$. By Lemma \ref{lemma:contr-filt}, we have $\ker(\omega)_k = P^k(\mathbb Z^{2g})$ for all $k \le g$. We will filter $F_r \coker(\omega)_{k-1}$ for $r \le \lfloor(k-1)/2 \rfloor$ and set 
\[F_r H_k(N_g;\mathbb Z) = \begin{cases} F_r \coker(\omega)_{k-1}, & r \le \lfloor \frac{k-1}{2} \rfloor \\ H_k(N_g;\mathbb Z). & r > \lfloor \frac{k-1}{2} \rfloor \end{cases}\]
The subquotients of this filtration are then precisely $P^k$ and the subquotients of $\coker(\omega)_{k-1}$. 

By Propositions \ref{prop:splitting-firsthalf} and \ref{prop:splitting-across-midpoint}, one may choose $\mathbb Z$-splittings of $\Lambda^k(\mathbb Z^{2g})$ for $0 \le k \le g+1$ compatible with the action of $\iota_\omega$. Arguing as in the proof of Theorem \ref{thm:failure-of-hard-lefschetz}, for $k \le g$ this allows us to give $\coker(\omega)_{k-1}$ a $\mathbb Z$-split filtration whose subquotients are identified with the cokernel of the associated graded map to $\iota_\omega$. To complete the proof for $k \le g$, observe that by Corollary \ref{cor:Lefschetz-graded} we obtain for $0 \le r \le \lfloor \frac{k-1}{2} \rfloor$ \[\text{gr}_r \coker(\omega)_{k-1} \cong P^{k-2r-1}(\mathbb Z^{2g})/(r+1).\] 

For $k = g + j > g$, we first use Poincar\'e duality to identify $H_k(N_g;\mathbb Z) \cong H^{g+1-j}(N_g;\mathbb Z)$ as $\text{Sp}(2g)$-modules. Next, we use the universal coefficient theorem to obtain an $\text{Sp}(2g)$-equivariant short exact sequence 
\[0 \to \text{Ext}(H_{g-j}(N_g),\mathbb Z) \to H^{g+1-j}(N_g) \to \text{Hom}(H_{g+1-j}(N_g),\mathbb Z) \to 0.\] The final term is identified with the dual $P^{g+1-j}(\mathbb Z^{2g})^* = \text{Hom}(P^{g+1-j}(\mathbb Z^{2g}), \mathbb Z)$. 

We will use our existing filtration on $H_{g-j}(N_g)$ to define a filtration on $H^{g+1-j}(N_g)$. Writing $c = \lfloor (g-j-1)/2 \rfloor = \lfloor (2g-k-1)/2\rfloor$, we set
\[F_r H_{g+j}(N_g) \cong F_r H^{g+1-j}(N_g) = \begin{cases} \text{Ext}(H_{g-j}(N_g)/F_{c-r} H_{g-j}(N_g), \mathbb Z), & 0 \le r \le c \\ H^{g+1-j}(N_g). & r > c \end{cases} \]
This defines a $\mathbb Z$-split filtration because Ext is contravariant and the filtration on $H_{g-j}$ is $\mathbb Z$-split. The initial subquotients are then identified with 
\[\text{Ext}(\text{gr}_{c-r} H_{g-j}, \mathbb Z) \cong \text{Ext}\left(P^{g-j-2c+2r-1}(\mathbb Z^{2g})/(c-r+1), \mathbb Z\right).\]
Simplifying indexing, this is equal to $P^{2r+\epsilon}(\mathbb Z^{2g})^*/\left(\lfloor \frac{2g-k+1}{2} -r\rfloor\right).$

To conclude, observe that primitive subspaces are self-dual as $\text{Sp}(2g)$-modules: \[(x,y) \mapsto x \wedge y \wedge \omega_{g-k} \in \Lambda^{2g}(\mathbb Z^{2g}) \cong \mathbb Z\] defines an $\text{Sp}(2g)$-invariant perfect pairing on $P^k(\mathbb Z^{2g})$ for all $0 \le k \le g$.
\end{proof}

\begin{remark}
The awkward relationship between the proof for $k \le g$ and $k > g$ is a necessary consequence of Lemma \ref{lemma:no-total-filtration}. In fact, $\coker(\omega)_{k-1}$ disagrees with $\coker(\text{gr } \omega)_{k-1}$ for $k > g$. While the proof of \cite[Theorem 1.8]{LeePacker} is quite different than the proof presented above, the authors of that article also split the computation into halves and used Poincar\'e duality.
\end{remark}

\subsection{Comparing the actions of $\omega$ and $e^\omega-1$}
For our application to Heegaard Floer homology, the essential point is to compare the kernels and cokernels of the contraction action by $\omega$ and $e^\omega - 1 = \omega + \omega_2 + \cdots$. There is no difficulty comparing their kernels over the integers: 

\begin{lemma}\label{lemma:ker-contraction}
The kernels of the contraction action on $\Lambda^*(\mathbb Z^{2g})$ satisfy \[\ker(\omega) = \ker(e^\omega-1) = \bigoplus_{0 \le k \le g} P^k(\mathbb Z^{2g}).\]
\end{lemma}
\begin{proof}
Suppose $x \in \ker(\omega)$. Because $k!\iota_{\omega_k}(x) = \iota_{\omega}^k(x)$ and the integers are torsion-free, we have $\iota_{\omega_k}(x) = 0$ for all $k > 0$, so that $x \in \ker(e^\omega - 1)$. Conversely, suppose $x_0 + \cdots + x_n \in \ker(e^\omega-1)$ with $x_i \in \Lambda^i$; we will show $x_i \in \ker(\omega)$ for all $i$ by induction. The base case $i =n+1$ is trivial. Suppose $\iota_\omega(x_j) = 0$ for $j > i$. Inspecting the component of $\iota_{e^\omega-1}(x)$ in degree $i-2$, we see \[\iota_{\omega}(x_i) + \iota_{\omega_2}(x_{i+2}) + \cdots = 0.\] By the inductive hypothesis and our previous remarks, this equation simplifies to $\iota_\omega(x_i) = 0.$
\end{proof}

\begin{remark}
    Note that this argument fails over $\Lambda^*(\mathbb F_p^{2g})$, and the result is false in that case. 
\end{remark}

More interesting is to compare the \textit{cokernels}. Here we present three comparison results, two positive and one negative.

\begin{prop}\label{prop:computation-for-coker}
The cokernels of contraction by $\omega$ and $e^\omega-1$ satisfy the following: 
\begin{enumerate}[label=(\alph*)]
    \item\label{part:a} $\coker(\omega)$ and $\coker(e^\omega-1)$ are isomorphic abelian groups.
    \item\label{part:b} There exist $\textup{Sp}(2g)$-invariant filtrations on $\coker(\omega)$ and $\coker(e^\omega-1)$ whose associated graded $\textup{Sp}(2g)$-modules satisfy \[\textup{gr}_{k} \coker(e^\omega-1) \cong \textup{gr}_{k} \coker(\omega) \cong \begin{cases} \bigoplus_{r=0}^{k} P^{g-k}(\mathbb Z^{2g})/(r) & 0 \le k \le g \\ 0 & \textup{otherwise}. \end{cases} \]
    \item\label{part:c} For $g=4$, the cokernels of $\iota_\omega$ and $\iota_{e^\omega-1}$ on $\Lambda^*(\mathbb F_2^{2g})$ are non-isomorphic as $\textup{Sp}(2g;\mathbb F_2)$-modules, so $\coker(\omega)$ and $\coker(e^\omega-1)$ are not isomorphic as $\textup{Sp}(2g)$-modules. 
\end{enumerate}
\end{prop}

Because the torsion of $\oplus_k H_k(N_g;\mathbb Z)$ is the torsion of $\coker(\omega)$, comparing the formula from Proposition \ref{prop:computation-for-coker}\ref{part:b} to Theorem \ref{thm:LP}, we see that the given filtration on $\coker(\omega)$ is \textit{not} $\mathbb Z$-split. It is not clear to the authors whether $\coker(\omega)$ and $\coker(e^\omega-1)$ are isomorphic as filtered abelian groups.

The arguments for the three parts are completely unrelated, and we separate them into three subsections below. The third uses the calculations of Section \ref{sec:Lef} to give an explicit description of the two cokernels; showing that they are not isomorphic depends on an enumeration of the $\text{Sp}(8)$-invariant submodules of $\Lambda^2(\mathbb F_2^8)$ and $\Lambda^4(\mathbb F_2^8)$, which is carried out using MAGMA. 

\subsection*{Proof of Proposition \ref{prop:computation-for-coker}\ref{part:a}.}
By Hodge--Lefschetz duality, we may instead compute the cokernel of wedge product with $\omega$ and $e^\omega-1$. For $S \subset \{1, \cdots, 2g\}$ with elements $i_1 < \cdots < i_k$ listed in increasing order, write \[e^S = e^{i_1} \wedge \cdots e^{i_k}.\]

We use the \textit{pair-free subspace decomposition} \cite[Definition 1.2]{LeePacker}. Subsets $S$ of $\{1,2,...,2g\}$ are called pair-free if $S \cap \{2i-1,2i\}=\varnothing$, and \textit{pair-full} if $2i-1 \in S \iff 2i \in S$. Given a pair-free set, the corresponding \textit{pair-free subspace} is $$V(S) = \text{span}\{ e^S \wedge e^T \mid T \text{ pair-full}\}.$$ Each $V(S)$ is preserved under wedge product with $\omega$, and \cite[Equation (1.8)]{LeePacker} asserts that 
\[\Lambda^*(\mathbb Z^{2g})=\bigoplus_{S \text{ pair-free}}V(S).\] 
Finally, contraction by $e^S$ defines an isomorphism $V(S) \cong V(\varnothing) \subset \Lambda^*(\mathbb Z^{2(g-|S|)})$ which respects the action of $\omega$. Therefore, it suffices to show that the cokernels of $\omega$ and $e^\omega-1$ on $V(\varnothing) \subset \Lambda^*(\mathbb Z^{2k})$ are isomorphic. 

Identify $V(\varnothing) \subset \Lambda^*(\mathbb Z^{2k})$ with the commutative ring $R_k = \mathbb Z[v_1,...,v_k]/(v_i^2),$ where $v_i = e^{2i-1} \wedge e^{2i}$ and thus $\omega = \sum_{i=1}^k v_i$. One may define a ring homomorphism $\phi: R_k \to R_k$ by specifying a square-zero element $\phi(v_i) \in R_k$ for all $k$. We define this recursively by \[\phi(v_i) = \begin{cases} v_1 & i = 1 \\ v_i(1 + \phi(v_{i-1}))  & 1 < i \le k. \end{cases}\] We have $\phi(v_i)^2 = 0$ because $v_i^2 = 0$, and induction on $k$ shows both that $\phi$ is an isomorphism and $\phi(\omega) = e^\omega - 1$. Because $\phi(\omega x) = \phi(\omega) \phi(x) = (e^{\omega} - 1) \phi(x),$ we see that $\phi$ induces an isomorphism $\coker(\omega) \cong \coker(e^\omega - 1)$.

\subsection*{Proof of Proposition \ref{prop:computation-for-coker}\ref{part:b}.}

We compare these using the \textit{shifted Lefschetz filtration} \[\mathcal F_k \Lambda = \bigoplus_{r} F_r \Lambda^{g-k+2r}(\mathbb Z^{2g}).\] The sequences $\mathcal F_{2k}$ and $\mathcal F_{2k+1}$ define increasing filtrations of $\Lambda^{\text{even}}$ and $\Lambda^{\text{odd}}$, respectively, with associated graded $\mathcal{G}r_{k} \Lambda \cong \oplus_{r=0}^k P^{g-k}(\mathbb Z^{2g})$. The contraction action of $\omega$ sends $\mathcal F_k$ into $\mathcal F_k$, and we filter $\coker(\omega)$ by \[\mathcal F_k \coker(\omega) = \frac{\mathcal F_k \Lambda}{\text{im}(\iota_\omega) \cap \mathcal F_k \Lambda}.\]

\begin{lemma}\label{lemma:specseq-collapse}
Restricting $\iota_\omega$ to $\Lambda^{\textup{even}}$ or $\Lambda^{\textup{odd}}$ depending on the parity of $k$, we have $\iota_\omega^{-1}(\mathcal F_k \Lambda) = \mathcal F_k \Lambda$ and so $\textup{im}(\iota_\omega) \cap \mathcal F_k \Lambda = \iota_\omega \mathcal F_k \Lambda$. The analogous result holds also for $\iota_{e^\omega-1}$. 
\end{lemma}
\begin{proof}
The case of $\iota_\omega$ is simpler, because the action breaks up into direct summands. The claim asserts that $\iota_\omega^{-1}(F_r \Lambda^{g-k+2r}) = F_{r+1} \Lambda^{g-k+2(r+1)}$, which was established in Lemma \ref{lemma:contr-filt}.

Suppose now that the even-degree form $x = x_0 + \cdots + x_{2g} \in \Lambda^*$ has $\iota_{e^\omega-1}(x) \in \mathcal F_{g-k}$ for $k = 2\ell$ even; the odd case will be similar. We will show $x \in \mathcal F_{g-2\ell} = \oplus_r F_r \Lambda^{2\ell+2r} = \oplus_r F_{j-\ell} \Lambda^{2j}$ as well. Suppose we know $x_{2i} \in  F_{i-\ell} \Lambda^{2i}$ for $i > r$. Then by assumption we have \[\iota_\omega(x_{2r}) + \iota_{\omega_2}(x_{2r+2}) + \cdots \in F_{r-\ell-1} \Lambda^{2r-2}.\] By inductive hypothesis $x_{2r+2j} \in F_{r+j-\ell} \Lambda^{2r+2j}$, so that $\iota_{\omega_{j+1}}(x_{2r+2j}) \in F_{r-\ell-1} \Lambda^{2r-2}$. We obtain $\iota_\omega(x_{2r}) \in F_{r-\ell-1} \Lambda^{2r-2}$, so by Lemma \ref{lemma:contr-filt} we have $x_{2r} \in F_{r-\ell} \Lambda^{2r}$. This completes the induction.
\end{proof}

It follows that 
\begin{align*}\mathcal Gr_k \coker(\omega) &= \frac{\mathcal F_k \Lambda}{\textup{im}(\iota_\omega) \cap \mathcal F_k \Lambda}\Bigg/ \frac{\mathcal F_{k-1} \Lambda}{\textup{im}(\iota_\omega) \cap \mathcal F_{k-1} \Lambda} = \frac{\mathcal F_k \Lambda}{\iota_\omega \mathcal F_k \Lambda}\Bigg/ \frac{\mathcal F_{k-1} \Lambda}{\iota_\omega \mathcal F_{k-1} \Lambda} \\
&\cong \frac{\mathcal F_k \Lambda}{\mathcal F_{k-1} \Lambda + \iota_\omega \mathcal F_k \Lambda} \cong \frac{\mathcal F_k \Lambda}{\mathcal F_{k-1} \Lambda}\Bigg/ \iota_\omega\left(\frac{\mathcal F_k \Lambda}{\mathcal F_{k-1} \Lambda}\right) = \coker(\mathcal Gr_k \iota_\omega),
\end{align*}
and an analogous chain of equalities applies to $\iota_{e^\omega-1}$. It thus suffices to compute and compare the cokernels of the associated graded maps. 

Writing $\mathcal Gr_k \Lambda = \bigoplus_{r=0}^{k} P^{g-k} = P^{g-k} \otimes \mathbb Z^{k+1}$, Corollary \ref{cor:Lefschetz-graded} identifies the action of $\iota_\omega$ and $\iota_{e^\omega-1}$ on this associated graded module as $1 \otimes D_k$ and $1 \otimes E_k$, where $D_k$ and $E_k$ are the matrices \[D_k = \begin{pmatrix} 
    0 & 1 & 0 & \cdots & 0  \\ 
    0 & 0 & 2 & \cdots & 0  \\ 
    0 & 0 & 0 & \cdots & 0  \\ 
    \cdots & \cdots & \cdots & \cdots & \cdots \\ 
    0 & 0 & 0 & \cdots & k \\
    0 & 0 & 0 & \cdots & 0  \end{pmatrix}, \quad \quad E_k = \begin{pmatrix}
    0 & 1 & 1 & \cdots & 1  \\ 
    0 & 0 & 2 & \cdots & (g-k)  \\ 
    0 & 0 & 0 & \cdots & \binom{g-k}{g-k-2}  \\ 
    \cdots & \cdots & \cdots & \cdots & \cdots \\ 
    0 & 0 & 0 & \cdots & k \\
    0 & 0 & 0 & \cdots & 0 \end{pmatrix}.\]
That is, indexing the entries of $E_k$ as $a_{r,s}$ for $0 \le s,r \le k$, we have $a_{r,r+j} = \binom{r+j}{j}$ for $j \ge 1$.

The cokernel of the $D_k$ is clearly $\bigoplus_{r=0}^k \mathbb Z/k$, which completes the computation of $\mathcal Gr_{g-k} \coker(\omega)$. It remains to identify the cokernels of $D_k$ and $E_k$, which follows from the lemma below. Its proof was explained to the authors by Matthew Bolan on MathOverflow \cite{MO-Bolan}.

\begin{lemma}\label{lemma:matrix-comp}
There exists an isomorphism $\varphi: \mathbb Z^{k+1} \to \mathbb Z^{k+1}$ so that $E_k = \varphi^{-1} D_k \varphi$. 
\end{lemma}
\begin{proof}
Consider $\mathbb Z^{k+1}$ as the abelian group of polynomials of degree at most $k$, equipped with the standard basis of monomials. The operator $D_k$ is the matrix representation of the differentiation operator on this space. To prove the result, it suffices to find a sequence of monic integer polynomials $p_d(x)$ of degree $d$ for which \[\frac{d}{dx} p_d(x) = \sum_{i=1}^d \binom{d}{i} p_{d-i}(x),\] with $\varphi$ the corresponding change of basis matrix. We take the \textit{Touchard polynomials} $p_d(x) = \sum_{k=0}^d S(d,k) x^k,$ where $S(d,k)$ is the Stirling number of the second kind, which counts partitions of $\{1, \cdots, d\}$ into $k$ disjoint nonempty subsets. The coefficient of $x^k$ in the desired relation then reads \[(k+1) S(d, k+1) = \sum_{i=1}^{d-k} \binom{d}{i} S(d-i, k).\] This is the special case $(n,m,\ell) = (d, 1, k)$ of a standard identity, \cite[Equation (6.28)]{Knuth}.
\end{proof}

\subsection*{Proof of Proposition \ref{prop:computation-for-coker}\ref{part:c}.}

Throughout this section, we assume $g = 4$ and we work over $\Lambda^*(\mathbb F_2^8)$, taking advantage of the filtration $F_r \Lambda^k(\mathbb F_2^8) = F_r \Lambda^k(\mathbb Z^8) \otimes \mathbb F_2$ and suppressing $\mathbb F_2^8$ from notation. Our first goal is an explicit description of the cokernels. The odd case is simpler: 

\begin{lemma}\label{lemma:oddcalc}
There are $\textup{Sp}(8)$-equivariant isomorphisms 
\[\coker(\omega;\mathbb F_2)_{\textup{odd}} \cong 
    \coker(e^\omega-1;\mathbb F_2)_{\textup{odd}} \cong \Lambda^1 \oplus \Lambda^1 \oplus P^3.\]
\end{lemma}
\begin{proof}
We compare the images first. The image of $\Lambda^1$ is zero, while $\iota_\omega: \Lambda^3 \to \Lambda^1$ is surjective, so the first summand is in the image. The map $\iota_\omega: \Lambda^5 \to \Lambda^3$ surjects onto $P^3$ and has $\text{gr}_2 \Lambda^5 \to \text{gr}_1 \Lambda^3$ zero by Corollary \ref{cor:Lefschetz-graded}, so that its image is precisely $F_0 \Lambda^3 = P^3$. Thus, restricted to $\Lambda^1 \oplus \Lambda^3 \oplus \Lambda^5$, both maps have the same image, equal to $\Lambda^1 \oplus P^3$. We are reduced to comparing the cokernels of two maps \[\Lambda^7 \to \text{gr}_1 \Lambda^3 \oplus \Lambda^5 \oplus \Lambda^7,\] respectively $(0, \iota_\omega, 0)$ and $(\text{gr} \iota_{\omega_2}, \iota_\omega, 0)$. The map $\iota_{\omega}: \Lambda^7 = \text{gr}_3 \Lambda^7 \to \text{gr}_2 \Lambda^5$ is an isomorphism by Corollary \ref{cor:Lefschetz-graded}, so $\text{gr}_1 \Lambda^3 \oplus F_1 \Lambda^5 \oplus \Lambda^7$ is a complement to the image of both $\iota_\omega$ and $\iota_{e^\omega-1}$. This complementary subspace is isomorphic to $\Lambda^1 \oplus \Lambda^1 \oplus P^3$.

\end{proof}

The description of the even part of the cokernel requires some preliminaries. 

\begin{definition}
Write $C_4 = \Lambda^4/\iota_\omega(F_2 \Lambda^6)$, let $T \subset C_4$ be the image of $\iota_\omega: \text{gr}_3 \Lambda^6 \to C_4$, and let $C_4' \subset C_4$ be the kernel of $\iota_{\omega_2}: C_4 \to \Lambda^0$.
\end{definition}

\begin{lemma}\label{lemma:5parts} 
All of the following hold: 
\begin{enumerate}[label=(\alph*)]
\item When restricted to the domain $\Lambda^0 \oplus \Lambda^2 \oplus \Lambda^4 \oplus F_2 \Lambda^6$, the images of $\iota_\omega$ and $\iota_{e^\omega-1}$ both equal \[\Lambda^0 \oplus P^2 \oplus (\iota_\omega F_2 \Lambda^6) \oplus 0 \oplus 0 \subset \Lambda^0 \oplus \Lambda^2 \oplus \Lambda^4 \oplus \Lambda^6 \oplus \Lambda^8\]

\item The map $\iota_{\omega_2}: \textup{gr}_3 \Lambda^6 \to \textup{gr}_1 \Lambda^2$ is an isomorphism, while $\iota_{\omega_3}: \Lambda^8 \to \textup{gr}_1 \Lambda^2$ is zero. 
\item The map $j = \iota_\omega: \Lambda^8 \to \Lambda^6 \cong \Lambda^2$ is injective with $j(\omega_4) = \omega_3 \in F_2 \Lambda^6$.
\item We have $C_4 = C_4' \oplus T$.
\item The map $i = \iota_{\omega_2}: \Lambda^8 \to C_4$ has image $i(\omega_4) = \omega_2 \in C_4'$ and is nonzero.
\end{enumerate}
\end{lemma}
\begin{proof}
For the first claim, the maps are the same on $\Lambda^0 \oplus \Lambda^2$, and $\iota_\omega: \Lambda^2 \to \Lambda^0$ is surjective. The map $\iota_\omega: \text{gr}_1 \Lambda^4 \to P^2$ is surjective, while $\text{gr}_2 \Lambda^4 \to \text{gr}_1 \Lambda^2$ is zero, so $\iota_\omega(\Lambda^4) = P^2$. The map $\iota_{\omega_2}: \Lambda^4 \to \Lambda^0$ is irrelevant because $\Lambda^0$ is already contained in the image. The map $\iota_{\omega_3}$ vanishes on $F_2 \Lambda^6$ and $\iota_{\omega_2}$ lands in the summand $P^2$, which is already in the image, so the image of the first three summands is as described. 

The second claim follows from Corollary \ref{cor:Lefschetz-graded}. The third claim uses that calculation to see $j(\omega) \in F_2 \Lambda^6$ and the explicit formula comes from Lemma \ref{lemma:general-mn-formula}. Injectivity follows because $\omega_3$ is nonzero in $\mathbb F_2$. The fourth claim uses that $\iota_\omega: \text{gr}_3 \Lambda^6 \to \text{gr}_2 \Lambda^4$ and $\iota_{\omega_2}: \text{gr}_2 \Lambda^4 \to \Lambda^0$ are both isomorphisms, so $T$ is complementary to $\ker(\iota_{\omega_2})$. For the final claim, Lemma \ref{lemma:general-mn-formula} gives $i(\omega_4) = \omega_2$, while $\iota_{\omega_2}(\omega_2) = -3$ and $\iota_\omega(\omega_2) = -3\omega$. This gives us the explicit formula, the fact that $i(\omega) \in C'_4 = \ker(\iota_{\omega_2})$, and $i(\omega) \not\in i_\omega F_2 \Lambda^6$ as $i_\omega^2 = 0$ mod $2$.
\end{proof}

The above facts quickly lead to the following:

\begin{lemma}\label{lemma:evencalc}
There are $\textup{Sp}(8)$-equivariant isomorphisms 
\begin{align*}\coker(\omega;\mathbb F_2)_{\textup{even}} &\cong \mathbb F_2^2 \oplus C_4' \oplus \Lambda^2/\langle \omega \rangle \\
\coker(e^\omega-1;\mathbb F_2)_{\textup{even}} &\cong \mathbb F_2^2 \oplus \frac{C'_4 \oplus \Lambda^2}{\langle \omega_2, \omega \rangle}.
\end{align*} 
\end{lemma}
\begin{proof}
Eliminating the image of $\Lambda^0 \oplus \Lambda^2 \oplus \Lambda^4 \oplus F_2 \Lambda^6$, we are left with determining the cokernel of a pair of maps \[\text{gr}_3 \Lambda^6 \oplus \Lambda^8 \to \text{gr}_1 \Lambda^2 \oplus [T \oplus C_4'] \oplus \Lambda^6 \oplus \Lambda^8.\] Neither map has a component mapping to $\Lambda^8$, and both maps split as the direct sum of a pair of maps $\text{gr}_3 \Lambda^6 \to (\text{gr}_1 \Lambda^2, T)$ and $\Lambda^8 \to C_4' \oplus \Lambda^6$. The first map is either $(0,1)$ or $(1,1)$, hence the cokernel is $\mathbb F_2$ in both cases. The second map is either $(0,j)$ or $(i,j)$. Substituting $i(\omega_4) = \omega_2$, $j(\omega_4) = \omega_3$ and applying the duality isomorphism gives us the stated description.
\end{proof}

The Krull--Schmidt theorem implies that $\text{Sp}(8)$-modules are cancellative under direct sum. It follows that if $\coker(\omega;\mathbb F_2)$ and $\coker(e^\omega-1;\mathbb F_2)$ were isomorphic as $\text{Sp}(8)$-modules, we would have an isomorphism $$C_4' \oplus \Lambda^2/\langle \omega \rangle\cong \frac{C'_4 \oplus \Lambda^2}{\langle \omega_2, \omega\rangle}.$$ Because $\omega_2 \ne 0$ in $C'_4$, the latter $G$-module admits an injection from $\Lambda^2$, so this would imply the existence of an equivariant injection $\Lambda^2 \to C_4' \oplus \Lambda^2/\langle \omega \rangle$. It suffices to show that every $G$-module homomorphism from $\Lambda^2$ to $C_4'$ has $\omega$ in its kernel. We will find the following information useful on the way.

\begin{lemma}
We have $\dim_{\mathbb F_2} C_4 = 44$, so that $\dim_{\mathbb F_2} i_\omega F_2 \Lambda^6 = 26$.  
\end{lemma}
\begin{proof}
We have an abelian group isomorphism $\oplus_k H^k(N_4;\mathbb Z) \cong \ker(\omega;\mathbb Z) \oplus \coker(\omega;\mathbb Z)$, and by Theorem \ref{thm:LP} the former is $\mathbb Z^{252} \oplus (\mathbb Z/2)^{10}$, while $\ker(\omega) = \oplus_k P^k(\mathbb Z^8) \cong \mathbb Z^{126}$. It follows that $\coker(\omega;\mathbb Z) \cong \mathbb Z^{126} \oplus (\mathbb Z/2)^{10}.$ We thus have an isomorphism $\coker(\omega;\mathbb F_2) = \coker(\omega;\mathbb Z)/2 \cong \mathbb F_2^{136}.$
By Lemmas \ref{lemma:oddcalc} and \ref{lemma:evencalc}, we have $\dim \coker(\omega;\mathbb F_2) = \dim(C_4) + 92$ so $\dim C_4 = 44$.
\end{proof}

\begin{lemma}\label{lemma:computer-work}
The kernel of every $\textup{Sp}(8)$-equivariant homomorphism $\Lambda^2 \to C_4'$ contains $P^2$.
\end{lemma}
\begin{proof}
The submodules of $\Lambda^2$ and $\Lambda^4$ can be enumerated by the following simple MAGMA code: 
\begin{verbatim}
M := GModule(Sp(8,2));
L2 := ExteriorPower(M,2); 
L4 := ExteriorPower(M,4); 
Submodules(L2); 
Submodules(L4);
\end{verbatim}
The only non-trivial proper submodules of $\Lambda^2$ are $\langle \omega\rangle$ and $P^2$. If $f: \Lambda^2 \to C_4$ has kernel of dimension $d \in \{0, 1, 27, 28\}$, then $\pi^{-1}(f(\Lambda^2)) \subset \Lambda^4$ has dimension $26 + 28 - d = 54 - d$. Because $\Lambda^4$ has no submodules of dimension $53$ or $54$, we have $d \in \{27, 28\}$. Thus $f$ has kernel $P^2$ or $\Lambda^2$.
\end{proof}

\section{Background on Heegaard Floer homology and cup homology}\label{sec:HF-bg}
In this section, we review some general facts about Heegaard Floer homology and the definition of cup homology. Every fact used here can be found in one of \cite{OzSz3D1,OzSz3D2,OzSz4D1,OzSzKnot2,Mark,Zemke}, though see the end of the section for some remarks on subtleties in the literature.

\subsection*{Spin$^c$ structures}
If $Y$ is a closed oriented $3$-manifold, there is associated with $Y$ a set $\text{Spin}^c(Y)$ of `spin$^c$ structures on $Y$'; an element might be denoted $\mathfrak s$. This set carries some important additional structure: there is a simply transitive action of $H^2(Y;\mathbb Z)$ on $\text{Spin}^c(Y)$, here written $x \cdot \mathfrak s$, and a function \[c_1: \text{Spin}^c(Y) \to H^2(Y; \mathbb Z), \quad \quad c_1(x \cdot \mathfrak s) = 2x + c_1(\mathfrak s).\] The map $c_1$ surjects onto $2H^2(Y;\mathbb Z)$, and the formulas above immediately imply the following statement. 

\begin{lemma}\label{lemma:2-tors}
If $Y$ is a closed oriented $3$-manifold for which $H^2(Y;\mathbb Z)$ has no $2$-torsion, there is a unique spin$^c$ structure on $Y$ satisfying $c_1(\mathfrak s) = 0$. 
\end{lemma}

For the remainder of this article, we will assume that $H^2(Y; \mathbb Z)$ has no $2$-torsion.

A similar statement applies to compact oriented $4$-manifolds $W$. If $W_0 \subset W$ is a compact $4$-dimensional submanifold, there is a restriction map $\text{Spin}^c(W) \to \text{Spin}^c(W_0)$; similarly if $Y \subset \partial W$ is a boundary component of $W$. These restriction maps are compatible with $c_1$ and the actions of $H^2$.

Finally, there is an action on these sets by the group of orientation-preserving diffeomorphisms. The action by $H^2$ and the map $c_1$ are both compatible with the diffeomorphism action.

\subsection*{The functor $HF^\infty$}
A cobordism $W: Y \to Y'$ is a compact oriented $4$-manifold $W$ equipped with an orientation-preserving diffeomorphism $\varphi: \partial W \to -Y \sqcup Y'$. This diffeomorphism is typically suppressed from notation. Inverting $\varphi$, we obtain inclusion maps $i: Y \to W$ and $j: Y' \to W$. 

A `spin$^c$ cobordism' $(W, \mathfrak t): (Y, \mathfrak s) \to (Y', \mathfrak s')$ is a cobordism $W: Y \to Y'$ equipped with a spin$^c$ structure $\mathfrak t$ on $W$ satisfying $i^* \mathfrak t \cong \mathfrak s$ and $j^* \mathfrak t \cong \mathfrak s'.$

Given a pair $(Y, \mathfrak s)$, Heegaard Floer homology produces a $\mathbb Z/2$-graded $\mathbb F_2[U, U^{-1}]$-module $HF^\infty(Y, \mathfrak s; \mathbb F_2).$ If $(W, \mathfrak t): (Y, \mathfrak s) \to (Y', \mathfrak s')$ is a spin$^c$ cobordism, the machine produces a well-defined homogeneous module homomorphism \[(W, \mathfrak t)_*: HF^\infty(Y, \mathfrak s; \mathbb F_2) \to HF^\infty(Y', \mathfrak s'; \mathbb F_2).\] This map well-defined up to diffeomorphism, in the following sense. If $(W, \mathfrak t): (Y, \mathfrak s) \to (Y', \mathfrak s')$ is a spin$^c$ cobordism and $\Psi: W' \to W$ is an orientation-preserving diffeomorphism, equip $W'$ with the boundary diffeomorphism $\varphi' = \varphi \Psi|_{\partial W'}: \partial W' \to -Y \sqcup Y'$. Then $(W', \Psi^*\mathfrak t)_* = (W,  \mathfrak t)_*.$

The functoriality statement is as follows. Given a pair of cobordisms 
\[(W_1, \mathfrak t_1): (Y, \mathfrak s) \to (Y', \mathfrak s'), \quad \quad (W_2, \mathfrak t_2): (Y', \mathfrak s') \to (Y'', \mathfrak s'')\] the cobordisms $W_1$ and $W_2$ may be composed to obtain a cobordism $W: Y \to Y''$; however, in the formulation we have given, the spin$^c$ structures may be composed in multiple ways. The correct composition law is 
\begin{equation}\label{eq:compositionlaw}
(W_2, \mathfrak t_2)_* (W_1, \mathfrak t_1)_* = \sum_{\substack{\mathfrak t \in \text{Spin}^c(W) \\ \mathfrak t|_{W_i} \cong \mathfrak t_i}} (W, \mathfrak t)_*.
\end{equation}

The group $HF^\infty(Y, \mathfrak s; \mathbb F_2)$ also carries an $\mathbb F_2[U, U^{-1}]$-module action by the algebra $\Lambda^*(H_1(Y;\mathbb Z)/\text{Tors})$. This action respects cobordism maps in the following sense: if $\gamma \in H_1(Y; \mathbb Z)$ and $\gamma' \in H_1(Y'; \mathbb Z)$ satisfy $i_* \gamma = j_* \gamma' \in H_1(W; \mathbb Z),$ we have \begin{equation}\label{eq:gamma-action}(W, \mathfrak t)_*(\gamma \cdot x) = \gamma' \cdot (W, \mathfrak t)_*(x).\end{equation}

\subsection*{Mapping class group actions}
Suppose that $Y$ is a closed oriented $3$-manifold and $\mathfrak s_0$ is the unique spin$^c$ structure on $Y$ with $c_1(\mathfrak s_0) = 0$.

Consider now the following special class of spin$^c$ cobordisms. If $\psi: Y \to Y$ is an orientation-preserving diffeomorphism, define the spin$^c$ cobordism $W_\psi: (Y, \mathfrak s_0) \to (Y, \mathfrak s_0)$ to have underlying manifold $[0,1] \times Y$, with boundary diffeomorphism $\varphi$ given by \[\varphi = \psi \sqcup 1_{-Y}: \{1\} \times Y \sqcup \{0\} \times -Y  \to Y \sqcup -Y,\] and equipped with the unique spin$^c$ structure with $c_1(\mathfrak t) = 0$.

If $\psi_0$ and $\psi_1$ are isotopic via the path of diffeomorphisms $\psi_t$, then $(t, x) \mapsto (t, \psi_0^{-1} \psi_t(x))$ defines a diffeomorphism $[0,1] \times Y \to [0,1] \times Y$ which identifies the cobordisms $W_{\psi_1}$ and $W_{\psi_0}$. It follows that $\psi \mapsto W_\psi$ descends to a map from the \textit{oriented mapping class group} \[\textup{MCG}^+(Y) = \pi_0 \text{Diff}^+(Y)\] of orientation-preserving diffeomorphisms up to isotopy to the set of spin$^c$ cobordisms $(Y, \mathfrak s_0) \to (Y, \mathfrak s_0)$ up to oriented diffeomorphism.

It is straightforward to see that we have $W_{\psi} \circ W_{\psi'} = W_{\psi \circ \psi'}$, and that the composite cobordism admits only one spin$^c$ structure restricting to $\mathfrak t, \mathfrak t'$ on the respective cylinders. Thus, we obtain the following result.

\begin{lemma}\label{lemma:HF-has-MCG-action}
Let $\mathfrak s_0$ be the unique spin$^c$ structure with $c_1(\mathfrak s_0) = 0$. Then the cobordism maps $(W_\psi)_*$ enrich $HF^\infty(Y, \mathfrak s_0; \mathbb F_2)$ with an action of $\textup{MCG}^+(Y)$ by $\mathbb F_2[U, U^{-1}]$ module automorphisms.
\end{lemma}

\subsection*{Computations for surgeries}
Suppose $Y$ is a closed oriented $3$-manifold with $K \subset Y$ null-homologous. Given an integer $n$, the \textit{$n$-surgery on $K$}, denoted $Y_n(K)$, is obtained topologically by deleting a tubular neighborhood of $K$, picking a longitude $\ell_n$ of $K$ with $\text{link}(K, \ell_n) = n$, attaching $[0,1] \times D^2$ along a neighborhood of this curve in $\partial(Y \setminus N(K))$, and then capping off the remaining $2$-sphere with a $3$-ball.

Another definition is $4$-dimensional in nature. There is a thickened embedding $\psi_n: S^1 \times D^2 \to Y$ with $\psi_n(z, 0)$ the embedding of $K$, while $\psi_n(z, 1)$ gives the embedding of the longitude $\ell_n$. Now take $[0,1] \times Y$, and paste $D^2 \times D^2$ along $S^1 \times D^2$ to $\{1\} \times Y$ via the map $\psi_n$. Smoothing out the codimension $2$ corners, we obtain the \textit{surgery cobordism} $W_n(K): Y \to Y_n(K)$. 

We will use the `large surgeries exact triangle', a useful relationship between the Heegaard Floer homology of $Y, Y_0(K)$, and $Y_{\pm n}(K)$. We will first need to discuss spin$^c$ structures on these manifolds, as well as the surgery cobordism $W_n(K)$. The statements given here are derived from \cite[Theorem 9.19]{OzSz3D2} and recalled in more detail in \cite[Section 2]{JabukaMark}.

Let $\mathfrak s_0$ be the unique spin$^c$ structure on $Y$ with $c_1(\mathfrak s_0) = 0$. So long as $n$ is odd, $H^2(Y_n(K); \mathbb Z)$ also has no $2$-torsion, and we write $\mathfrak t_0$ for its unique spin$^c$ structure with $c_1(\mathfrak t_0) = 0$.

Mayer--Vietoris gives an isomorphism $H^2(Y_0(K); \mathbb Z) \cong H^2(Y; \mathbb Z) \oplus \mathbb Z$, and there is a unique spin$^c$ structure $\mathfrak u_k$ on $Y_0(K)$ with $c_1(\mathfrak u_k) = (0,2k);$ its restriction to $Y \setminus K$ is isomorphic to the restriction of $\mathfrak s_0$. 

Similarly, consider the Mayer--Vietoris isomorphism $H^2(W_n(K); \mathbb Z) \cong H^2(Y; \mathbb Z) \oplus \mathbb Z.$ Because this group has no $2$-torsion, spin$^c$ structures on $W_n(K)$ are determined by their first Chern class, and spin$^c$ structures which have $c_1(\mathfrak r)|_{Y \setminus K} = 0$ are in bijection with $\mathbb Z$. More precisely, there is a unique spin$^c$ structure $\mathfrak r_i$ with $c_1(\mathfrak r_i) = (0, 2i-1)$. 

Under these assumptions, and taking $n$ both large and odd, the surgery long exact sequence takes the following form:

\begin{equation}\label{eq:surgery-seq}\cdots \to \bigoplus_{\ell \in \mathbb Z} HF^\infty(Y_0(K), \mathfrak u_{n\ell}; \mathbb F_2) \to HF^\infty(Y, \mathfrak s_0; \mathbb F_2) \xrightarrow{F} HF^\infty(Y_{-n}(K), \mathfrak t_0; \mathbb F_2) \to \cdots\end{equation}

Here, the map $F = F_0 + F_1 + \cdots$ is a (finite) sum of cobordism maps $F_i = (W_n(K), \mathfrak r_i)_*$. Finally, we have the following facts about $F_0, F_1$, established in \cite[Theorem 2.3]{OzSzKnot2}.

\begin{lemma}\label{lemma:Fi-maps}
In the situation described above, for $n$ sufficiently large, the map $F_0$ is an isomorphism, and we have $F_1 = F_0 J$ for $J$ the `basepoint-swapping automorphism' of $HF^\infty(Y, \mathfrak s_0; \mathbb F_2)$. 
\end{lemma}

We will not discuss the definition of the basepoint-swapping map; in the one case we use this, the map has been computed in the literature.

\subsection*{Coefficient rings}
The construction of Heegaard Floer homology requires many choices, and it is not clear that one can produce well-defined groups (as opposed to well-defined `groups up to isomorphism', which do not form a category); compare the discussion of \cite{JTZ}. 

The work of \cite{Zemke} conclusively settles functoriality with $\mathbb F_2$ coefficients. The Heegaard Floer homology groups themselves can be defined with coefficients in an arbitrary commutative ring, and are well-defined up to isomorphism, but extending these to a well-defined functor is quite subtle. Though some partial progress has been achieved \cite{Gartner-HFZ}, understanding these cobordism maps in general is still an open problem. 

Nevertheless, it is \textit{expected} that $HF^\infty(Y, \mathfrak s; \mathbb Z)$ can be enhanced to a well-defined functor to $\mathbb Z/2$-graded $\mathbb Z[U, U^{-1}]$-modules, where cobordisms $(W, \mathfrak t): (Y, \mathfrak s) \to (Y', \mathfrak s')$ are further equipped with a \textit{homology orientation}, an orientation of the real vector space $H^1(W) \oplus H^+(W) \oplus H^1(Y')$ \cite[Definition 3.4.1]{KMBook}.

We state this as a postulate, and occasionally use it as an assumption in certain results which require this extended functoriality. 

\begin{post}\label{post:HF-works}
The Heegaard Floer functor $HF^\infty(Y, \mathfrak s; \mathbb Z)$ is well-defined when cobordisms are further equipped with a homology orientation, and all results stated above lift in the natural way to $\mathbb Z$ coefficients.
\end{post}

For the cobordisms $W_\psi$ discussed above, $H^+(W) = 0$ and inclusion of the outgoing boundary component induces a canonical isomorphism $H^1(Y') \cong H^1(W)$, so these have a canonical homology orientation (and it is straightforward to see that these compose correctly, applying a variation on the definition of \cite[Section 8.2]{Scaduto-Kh}). Thus, we obtain the following enrichment of Lemma \ref{lemma:HF-has-MCG-action}.

\begin{lemma}\label{lemma:HF-has-MCG-action-over-Z}
If Postulate \ref{post:HF-works} holds, then the cobordism maps $(W_\psi)_*$ enrich $HF^\infty(Y, \mathfrak s_0; \mathbb Z)$ with an action of $\textup{MCG}^+(Y)$ by $\mathbb Z[U, U^{-1}]$ module automorphisms.
\end{lemma}

\subsection{Cup homology}
We briefly recall the definition of the groups $HC(Y;R)$ for $R$ a commutative ring, defined in \cite[Conjecture 4.10]{OzSzPlumbed} and further studied by Mark \cite{Mark}, who named $HC_*(Y)$ the `cup homology' of $Y$. For torsion spin$^c$ structures, the spin$^c$ structure plays no role in the definition of cup homology and is suppressed from notation.

If $Y$ is a closed connected oriented $3$-manifold, then its \textit{triple cup product} is \[\cup^3_Y: \Lambda^3 H^1(Y; \mathbb Z) \to \mathbb Z, \quad \quad \cup^3_Y(\alpha \wedge \beta \wedge \gamma) = (\alpha \smile \beta \smile \gamma)[Y];\] one takes a triple of degree $1$ cohomology classes, takes their wedge product, and evaluates the result on the fundamental class of $Y$. We denote the corresponding $3$-form by \[\omega_Y \in (\Lambda^3 H^1(Y; \mathbb Z))^* \cong \Lambda^3 H^1(Y; \mathbb Z)^* \cong \Lambda^3 \big(H_1(Y; \mathbb Z)/\text{Tors}\big).\] 

Contraction against the triple cup product defines a map $\iota_\cup: \Lambda^k H^1(Y;\mathbb Z) \to \Lambda^{k-3} H^1(Y; \mathbb Z)$ by \[\iota_\cup(\alpha_1 \wedge \cdots \wedge \alpha_k) = \sum_{1 \le i_1 < i_2 < i_3 \le k} (-1)^{i_1 + i_2 + i_3} \cup^3_Y(\alpha_{i_1}, \alpha_{i_2}, \alpha_{i_3}) \alpha_1 \wedge \cdots \hat \alpha_{i_1} \wedge \cdots \wedge \hat \alpha_{i_2} \wedge \cdots \wedge \hat \alpha_{i_3} \wedge \cdots \wedge \alpha_k.\] 

Cup homology $HC_*(Y;R)$ is the $\mathbb Z/2$-graded $R[U, U^{-1}]$-module given by the homology groups \[HC_*(Y;R) = H\left(\left(\Lambda^* (H^1(Y; \mathbb Z) \otimes_{\mathbb Z} R)\right)[U, U^{-1}],  \quad d(\alpha U^k) = \iota_\cup(\alpha) U^{k-1}\right).\]

Cup homology also defines a functor on an appropriate cobordism category, with cobordism maps zero for $b^+(W) > 0$. The cobordism maps in cup cohomology are described in \cite[Section 3.2]{LME2} up  to a sign ambiguity. These cobordism maps can be understood as follows. Considering the tori $\mathbb T_Y = H^1(Y; \mathbb R)/H^1(Y; \mathbb Z)$, we have a diagram \[\mathbb T_Y \xleftarrow{i^*} \mathbb T_W \xrightarrow{j^*} \mathbb T_{Y'}.\] A homology orientation on $W$ allows one to give a well-defined sign to the \textit{umkehr map} \[i_!: \Lambda^*(H^1(Y; \mathbb Z)) \to \Lambda^{*+b_1(W) - b_1(Y)}(H^1(W; \mathbb Z)),\] and the induced map $W_*$ on cup chain complexes is given by $j^* i_!$; that this defines a chain map is discussed in \cite[Remark 3.4]{LME2}, and comes down to the fact that $i^* \omega_Y = j^* \omega_{Y'}$. 

Once again, the existence of cobordism maps gives us an action of the mapping class group. 

\begin{lemma}\label{lemma:HC-has-MCG-action}
The cobordism maps $(W_\psi)_*$ enrich $HC_*(Y; R)$ with an action of $\textup{MCG}^+(Y)$ by $R[U, U^{-1}]$ module automorphisms.
\end{lemma}

Instead of understanding these in full generality, we investigate the induced map of the cobordisms $W_\psi$ discussed in the previous section. To state the following result, observe that $\psi \in \text{MCG}^+(Y)$ determines a map $\psi^*: H^1(Y; \mathbb Z) \to H^1(Y; \mathbb Z)$ which preserves the triple cup product, and hence defines a pullback map $\psi^*: HC_*(Y; R) \to HC_*(Y; R)$ satisfying $(\psi_1 \psi_2)^* = \psi_2^* \psi_1^*$. Thus, $\psi^*$ defines a \textit{right} action of the mapping class group on $HC_*(Y;R)$. 

\begin{lemma}\label{lemma:HC-MCG-formula}
The cobordism map $(W_\psi)_*$ described above coincides with $(\psi^{-1})^*$.
\end{lemma}
\begin{proof}
The inclusion map $i: Y \to [0,1] \times Y$ is $i(y) = (0, y)$, so $i_!$ is the identity. The map $j: Y \to [0,1] \times Y$ is given by $j(y) = (1, \psi^{-1}(y))$, and $j^* = (\psi^{-1})^*$. 
\end{proof}

\section{Computation for $\Sigma_g \times S^1$}\label{sec:compute-MCG}

\subsection{The mapping class group of $\Sigma_g \times S^1$}

Observe that if $f: \Sigma_g \to S^1$ is any smooth map, then the map \[\Phi_f: \Sigma_g \times S^1 \to \Sigma_g \times S^1, \quad \quad \Phi_f(x,z) = (x, f(x) \cdot z)\] is an orientation-preserving diffeomorphism. A map homotopic to $f$ gives rise to a diffeomorphism isotopic to $\Phi_f.$. Further, we have $\Phi_f \circ \Phi_g = \Phi_{f \cdot g}$. Because the abelian group $[\Sigma_g, S^1]$ of homotopy classes is isomorphic to the group $H^1(\Sigma_g; \mathbb Z)$, this gives rise to a homomorphism $\Phi: H^1(\Sigma_g; \mathbb Z) \to \text{MCG}^+(\Sigma_g \times S^1)$. 

On the other hand, if $\psi: \Sigma_g \to \Sigma_g$ is any diffeomorphism (not necessarily orientation-preserving), we may extend this to the orientation-preserving diffeomorphism \[L_\psi: \Sigma_g \times S^1 \to \Sigma_g \times S^1, \quad \quad L_\psi(x, z) = (\psi(x), z^{\pm 1}),\] where the sign in $z^{\pm 1}$ is the sign of the diffeomorphism $\psi$. This gives rise to a homomorphism $L: \text{MCG}(\Sigma_g) \to \text{MCG}^+(\Sigma_g \times S^1)$. Using the pullback action of $\text{MCG}(\Sigma_g)$ on $H^1(\Sigma_g; \mathbb Z)$, we may form the semidirect product of these two groups. The following result is well-known. 

\begin{lemma}\label{lemma:MCG-calc}
The maps $\Phi$ and $L$ assemble into an isomorphism \[\textup{MCG}(\Sigma_g) \ltimes H^1(\Sigma_g; \mathbb Z) \cong \textup{MCG}^+(\Sigma_g \times S^1).\]
\end{lemma}
This is a consequence of \cite[Corollary 7.5]{Waldhausen}, which identifies the mapping class group of a Haken manifold as $\text{Out}(\pi_1)$. A more geometric approach is suggested in the remark of \cite[page 85]{Waldhausen}; the mapping class group of $\Sigma_g \times S^1$ can be computed as the group of diffeomorphisms which preserve the fibration by circles, modulo isotopy through such diffeomorphisms. 

Write  $B(0,0) \subset S^2 \times S^1$ is the image of the third component of the Borromean rings after performing $0$-surgery on the first two. Then the $3$-manifold $\Sigma_g \times S^1$ is obtained as $0$-surgery on  $\#^g B(0,0) \subset \#^{2g} S^2 \times S^1$. After performing this surgery, the knot $\#^g B(0,0)$ with its zero-framing is sent to a fiber $K = \{\ast\} \times S^1$ with the trivial framing, so that a framed push-off is a nearby fiber. It will be important to know that many diffeomorphisms of $\Sigma_g \times S^1$ lift to $\#^{2g} S^2 \times S^1$ and to other surgeries on the same knot. 

\begin{definition}\label{def:MCG-pp}
We define \[\textup{MCG}^{++}(\Sigma_g \times S^1) \subset \textup{MCG}^+(\Sigma_g \times S^1)\] to be the group of mapping classes $\psi$ which preserve the isotopy class of $\{\ast\} \times S^1$. 
\end{definition}

It is clear from Lemma \ref{lemma:MCG-calc} that $\text{MCG}^{++}$ is an index-two subgroup of the full mapping class group, corresponding to the semidirect product $\text{MCG}^+(\Sigma_g) \ltimes H^1(\Sigma_g; \mathbb Z)$. We will focus our computations on this subgroup due to its lifting properties with respect to surgery on the knot $K=\{*\} \times S^1.$

\begin{lemma}\label{lemma:std-form}
Every $\psi \in \textup{MCG}^{++}(\Sigma_g \times S^1)$ is isotopic to a diffeomorphism which is the identity on a neighborhood $D^2 \times S^1$ of $\{\ast\} \times S^1$.
\end{lemma}
\begin{proof}
Let $U$ be a small neighborhood of $\ast \in \Sigma_g$. If $f: \Sigma_g \to S^1$ is any continuous map, there is a homotopic map $f': \Sigma_g \to S^1$ with $f'(D^2) = 1$. If $\psi: \Sigma_g \to \Sigma_g$ is any orientation-preserving diffeomorphism, $\psi$ is isotopic to a diffeomorphism $\psi'$ which is the identity on $D^2$. The diffeomorphism $L_{\psi'}\Phi_f' (x,z) = (\psi'(x), f'(x)z)$ is then isotopic to $L_\psi \Phi_f$ and of the desired form.
\end{proof}

This gives the following lifting procedure; notation in its statement and proof follows Section \ref{sec:HF-bg}.

\begin{cor}
If $\phi: \Sigma_g \times S^1$ is a diffeomorphism which is the identity in a neighorhood of $K = \{\ast\} \times S^1$, then $\phi$ lifts to an orientation-preserving diffeomorphism $\Phi_n: W_n(K) \to W_n(K)$ which restricts to $\phi$ on $\Sigma_g \times S^1$ and satisfies $\Phi_n^* \mathfrak r_i = \mathfrak r_i$. 
\end{cor}
\begin{proof}
Because $\phi$ restricts to the identity on the solid torus, it extends by the identity over $D^2 \times D^2$ to a self-diffeomorphism $\Phi_n$ of $W_n(K)$ equal to $\phi$ at the incoming end. 

Because $\Phi_n$ acts by the identity on $D^2 \times D^2$, it follows that under the Mayer--Vietoris isomorphism $H^2(W_n) \cong H^2(Y) \oplus \mathbb Z$ we have $\Phi_n^* = (\phi^*, 1)$, and therefore \[c_1(\Phi_n^* \mathfrak r_i) = \Phi_n^* c_1(\mathfrak r_i) = \Phi_n^* (0, 2i-1) = (0, 2i-1).\] Spin$^c$ structures on $W_n(K)$ are uniquely determined by their first Chern class, so $\Phi_n^* \mathfrak r_i = \mathfrak r_i$. 
\end{proof}

\begin{remark}
If $\psi \in \text{MCG}^+ \setminus \text{MCG}^{++}$, then there is an isotopic $\phi$ which is equal to $\phi(z, w) = (\bar z, \bar w)$ on $D^2 \times S^1$. It can be seen that $\phi$ still extends over $W_n(K)$, but we instead have $\Phi_n^* \mathfrak r_i = \mathfrak r_{1-i}$. This makes it cumbersome to compute and describe the full mapping class group action on $\text{HF}^\infty(\Sigma_g \times S^1, \mathfrak s_0; \mathbb Z)$, and is why we restrict to $\text{MCG}^{++}$. 
\end{remark}


\subsection{Computing $HC_*$}
We will now identify the action of $\textup{MCG}^{++}(\Sigma_g \times S^1)$ on $HC^*(\Sigma_g \times S^1; R)$. This is equivalent to identifying this $R[U]$-module and determining the actions of $\text{MCG}^+(\Sigma_g)$ and $H^1(\Sigma_g; \mathbb Z)$; these relate to the kernel and cokernel of contraction by $\omega$ on $\Lambda^*(\mathbb Z^{2g})$, studied in Section \ref{sec:applications}. 

\begin{prop}\label{prop:HC-MCG-calc}
For any commutative ring $R$, there is an isomorphism of $R[U]$-modules 
\begin{equation}\label{eq:HC-calc}HC_*(\Sigma_g \times S^1;R) \cong \left(\coker(\omega;R) \oplus e^0 \ker(\omega;R)\right)[U, U^{-1}].\end{equation}
With respect to this isomorphism, the action of $\textup{MCG}^+(\Sigma_g)$ factors through the inverse of the pullback action of $\textup{Sp}(2g)$ on $H^1(\Sigma_g; \mathbb Z)$. The action of $f \in H^1(\Sigma_g; \mathbb Z)$ is the identity on $\coker(\omega;R)$, and on $\ker(\omega;R)$ its second factor is the identity and its first factor coincides with the composite map \[\ker(\omega;R) \hookrightarrow \Lambda^*(R^{2g}) \xrightarrow{-f \wedge} \Lambda^*(R^{2g}) \twoheadrightarrow \coker(\omega;R).\]
\end{prop}

\begin{proof}
Decompose \[\Lambda^*(R^{2g+1}) \cong \Lambda^*(R^{2g}) \oplus e^0 \Lambda^*(R^{2g});\] then the kernel of contraction with $e^0 \omega U$ is \[\Lambda^*(R^{2g})[U, U^{-1}] \oplus e^0 \ker(\omega;R)[U, U^{-1}],\] while the image of contraction with $e^0 \omega$ is $\text{im}(\omega;R) \oplus 0$. It follows that 
\begin{align*}
HC_*(S^1 \times \Sigma_g;R) &\cong \left(\frac{\Lambda^*(R^{2g}) \oplus e^0 \ker(\omega;R)}{\text{im}(\omega;R) \oplus 0}\right)[U, U^{-1}] \\
&\cong \left(\coker(\omega;R) \oplus e^0 \ker(\omega;R)\right)[U, U^{-1}].
\end{align*}

Lemma \ref{lemma:HC-MCG-formula} computes that a diffeomorphism $\phi$ acts as $(\phi^{-1})^* = (\phi^*)^{-1}$, the pullback being the standard linear pullback on cohomology. For $\psi \in \text{MCG}^+(\Sigma_g)$, we have $L(\psi)^*(e^0) = e^0$ while the action of $L(\psi)^*$ on $H^1(\Sigma_g; \mathbb Z)$ defines an element of $\text{Sp}(2g)$. For $f \in H^1(\Sigma_g;\mathbb Z)$, we compute on $H_1(\Sigma_g \times S^1)$ that $\Phi_f(e_0) = e_0$ and $\Phi_f(\gamma) = \gamma + f(\gamma) e_0$ for $\gamma \in H_1(\Sigma_g)$. Dualizing, we obtain \[\Phi_f^*(e^0) = e^0 + f, \quad \Phi_f^*(c) = c\] for $c \in H^1(\Sigma_g;\mathbb Z)$. Taking inverses merely changes the first term to $(\Phi_f^{-1})^*(e^0) = e^0 - f$. Passing to the full exterior algebra, the cokernel term is fixed, while $e^0 \wedge x$ is sent to $e^0 \wedge x - f\wedge x$.
\end{proof}

\begin{remark}
The statement for the full mapping class group is as follows. If $\psi \in \textup{MCG}(\Sigma_g)$, then $\psi_* \in GL(2g, \mathbb Z)$ defines an element of the \textit{symplectic similitude group} $\text{GSp}(2g, \mathbb Z)$ of matrices with $A^* \omega = \pm \omega$, and the action factors through an action of $\text{GSp}$. The action on $\ker(\omega)$ is the standard pushforward action, while the action on $e^0 \coker(\omega)$ has a sign: $A \cdot [x] = (-1)^{\epsilon(A)} [Ax]$, where $A \omega = (-1)^{\epsilon(A)} \omega$. 
\end{remark}

\subsection{Computing $HF^\infty$}
The main result of this section is the following analogue of Proposition \ref{prop:HC-MCG-calc} for $HF^\infty$. The main technical input is the Heegaard Floer surgery sequence \eqref{eq:surgery-seq}, invariance under diffeomorphisms, and rigidity of a certain Heegaard Floer group. The argument is based on the calculation in \cite{JabukaMark}, modified to keep track of mapping class group actions.

\begin{prop}\label{prop:HF-MCG-calc}
There is a short exact sequence of $\mathbb F_2[U]$-modules with $\textup{MCG}^{++}(\Sigma_g \times S^1)$-action \[0 \to \coker(e^{\omega U} - 1; \mathbb F_2) \to HF^\infty(\Sigma_g \times S^1, \mathfrak s_0; \mathbb F_2) \to \ker(e^{\omega U} - 1; \mathbb F_2) \to 0,\] where the action of $\textup{MCG}^+(\Sigma_g)$ on the outer terms factors through the inverse of the pullback action of $\textup{Sp}(2g)$ and the action of $H^1(\Sigma_g; \mathbb Z)$ on the outer terms is trivial. Provided Postulate \ref{post:HF-works} holds and $g \ge 3$, the same is true with coefficients in $\mathbb Z$. 
\end{prop}

\begin{remark}
    In the case $g = 2$, the statement is that the action on the outer groups factors through a $\mathbb Z/2$-extension of $\text{Sp}(4)$; we cannot rule out the presence of an additional sign in Lemma \ref{lemma:compute-on-Y} below. It is likely that this hypothesis can be removed.
\end{remark}

We apply the surgery sequence to $(Y, K) = (\#^{2g} S^2 \times S^1, \#^g B(0,0))$. The manifold $Y_0(K)$ obtained by $0$-surgery on $K$ is diffeomorphic to $\Sigma_g \times S^1$ by a diffeomorphism sending the knot to a circle $\{x\} \times S^1$ with trivial framing. Ozsv\'ath and Szab\'o prove \cite[Theorem 9.3]{OzSzKnot1} that the group $HF^\infty(\Sigma_g \times S^1, \mathfrak s_t)$ is zero for all $|t| \ge g$, so taking $n \ge g$ the direct sum term in \eqref{eq:surgery-seq} collapses to $HF^\infty(\Sigma_g \times S^1, \mathfrak s_0)$. Recall that the $F$ term in the exact triangle simplifies to $F = F_0 + F_1$, where each map is the induced map of a cobordism $F_i = (W_n(K), \mathfrak r_i)_*$, and indeed every map in the exact triangle is a sum of cobordism maps. The sequence takes the form 
\begin{equation}\label{simplified-surgseq}\to \cdots HF^\infty(Y_0(K), \mathfrak u_0; \mathbb F_2) \to HF^\infty(Y, \mathfrak s_0; \mathbb F_2) \to HF^\infty(Y_{-n}(K), \mathfrak t_0; \mathbb F_2) \to \cdots
\end{equation}

Recall once more that we take $n$ to be odd.

\begin{lemma}\label{lemma:first-action}
Suppose $\psi: Y \to Y$ is a diffeomorphism equal to the identity in a neighborhood of $K$. Then $\psi_*$ induces a well-defined automorphism of the short exact sequence \eqref{simplified-surgseq}.
\end{lemma}

Here, an automorphism of a short exact sequence is an automorphism of each term which commutes with the connecting homomorphisms. We emphasize that Lemma \ref{lemma:first-action} does not state that $\psi_*$ depends only on the isotopy class of $\psi$, only that the mapping from diffeomorphisms to automorphisms is well-defined. 

\begin{proof}
Because $Y$ is connected and $\psi$ is the identity on an open set, $\psi$ is an orientation-preserving diffeomorphism. Because $\psi$ is the identity on a neighborhood of $K$, it extends to a diffeomorphism $\psi_r$ of any surgery $Y_r(K)$. The relevant spin$^c$ structures all have $c_1 = 0$, and the relevant $3$-manifolds all have no $2$-torsion in their second cohomology, so $\psi_r$ preserves the relevant spin$^c$ structures and therefore determines an automorphism of the corresponding Heegaard Floer homology groups. 

To see that that these automorphisms commute with the connecting maps in the surgery sequence, recall that invariance under diffeomorphisms means that if $W_n: Y \to Y_n(K)$ is a handle-attachment cobordism with spin$^c$ structure $\mathfrak r$ and $\Phi_n: W_n \to W_n$ is a diffeomorphism restricting to $\psi, \psi_n$ on the respective ends, and for which $\Phi_n^* \mathfrak r = \mathfrak r$, then $(W_n, \mathfrak s)_* \psi_* = (\psi_n)_* (W_n, \mathfrak r)_*$. All cobordism maps in the triangle are handle attachment cobordisms, and $\Phi_n$ can be defined by extending $\psi$ as the identity over the added $2$-handle.
\end{proof}

Now by Lemma \ref{lemma:std-form}, every element of $\textup{MCG}^{++}(\Sigma_g \times S^1)$ is isotopic to one of the form $\psi = \Phi_f L_\phi$ for some $\phi \in \textup{MCG}^+(\Sigma_g)$ which fixes a neighborhood of the basepoint and some $f: \Sigma_g \to S^1$ which is constant in a neighborhood of the basepoint. This induces a diffeomorphism $\hat \psi: Y \to Y$ which is the identity on $K$, where again $Y = \#^{2g} S^2 \times S^1$. 

Ozsv\'ath and Szab\'o prove \cite[Theorem 9.3]{OzSzKnot1} that \begin{equation}\label{eq:startingpoint}
    HF^\infty(Y,\mathfrak s_0; \mathbb Z)\cong \Lambda^*(\mathbb Z^{2g})[U, U^{-1}] \cong \Lambda^*(H^1(\Sigma_g; \mathbb Z))[U, U^{-1}].\end{equation} 

\begin{lemma}\label{lemma:compute-on-Y}
Suppose $\hat \psi: Y \to Y$ is a diffeomorphism constructed as above, beginning with $\psi = \Phi_f L_\phi: \Sigma_g \times S^1 \to \Sigma_g \times S^1$. After applying the isomorphism of \eqref{eq:startingpoint} and passing to $\mathbb F_2$ coefficients, the action of $\hat \psi: Y \to Y$ is taken to $(\phi^{-1})^*$, the pullback action of $\phi^{-1}$. Provided Postulate \ref{post:HF-works} holds and $g \ge 3$, the same is true over $\mathbb Z$.
\end{lemma}
\begin{proof}
We first identify the map induced by $\hat \psi$ on $H^1(Y; \mathbb Z)$. The Mayer--Vietoris sequence identifies $H^1(Y; \mathbb Z) \cong H^1(\Sigma_g; \mathbb Z) \subset H^1(\Sigma_g \times S^1; \mathbb Z)$, the summand of $1$-forms which vanish on $H_1(S^1)$; because $\psi$ and $\hat \psi$ are both the identity on a neighborhood of the knot, their actions are compatible with the Mayer--Vietoris sequence, and we find that the action of $\hat \psi$ on $H^1(\Sigma_g; \mathbb Z)$ is the restriction of the action of $\psi$. Because $\psi = \Phi_f L_\phi$ fixes a neighborhood of the knot we must have $f = 0 \in H^1(\Sigma_g; \mathbb Z)$. Because $L_\phi(x, z) = (\phi(x), z)$, we finally see that $\hat \psi^* = \phi^*$. Considering the cobordism $W_{\hat \psi}: Y \to Y$, the inclusion $i: Y \to W$ induces the identity on homology, whereas $j: Y \to W$ induces $\hat \psi^{-1}$ on homology, which is the same as the map induced by $\phi^{-1}$. It follows from \eqref{eq:gamma-action} that for $f = (W_{\hat \psi})_*$, we have \[f(\gamma \cdot x) = (\phi_* \gamma) \cdot f(x).\]

Though it is only stated for $\phi_* = 1$, the argument of \cite[Lemma 5.3]{OzSzKnot2} shows that there is a unique such automorphism $f$ up to sign. Because $(\phi^{-1})^*$ satisfies this formula, we see $f = \pm (\phi^{-1})^*$.

To pin down the sign over $\mathbb Z$, what we have described gives an action on $H^1(\Sigma_g; \mathbb Z)$ of the mapping class group $\text{MCG}^+(\Sigma_{g,1})$ of the genus $g$ surface with one boundary component. It differs from the action of $(\phi^{-1})^*$ by a homomorphism $\epsilon: \text{MCG}^+(\Sigma_{g,1}) \to \pm 1$, but this mapping class group has trivial abelianization for $g \ge 3$ \cite[Theorem 5.1]{MCG-ab} so $\epsilon = 1$. 
\end{proof}

In particular, the action of $\hat \psi$ on this term depends only on the isotopy class of $\psi \in \textup{MCG}^{++}(\Sigma_g \times S^1)$; what's more, the action of the subgroup of transvections is trivial on $Y$. 

\begin{cor}
The action by diffeomorphisms of Lemma \ref{lemma:first-action} descends to an action of $\textup{MCG}^{++}(\Sigma_g \times S^1)$ on the surgery exact triangle with $\mathbb F_2$ coefficients. If Postulate \ref{post:HF-works} holds and $g \ge 3$, the same is true with $\mathbb Z$ coefficients. 
\end{cor}
\begin{proof}
That the induced map on the $\Sigma_g \times S^1$ term depends on $\psi$ only up to isotopy was discussed in the Section \ref{sec:HF-bg}. That this is true on the $Y$ term follows from Lemma \ref{lemma:compute-on-Y}. Ozsv\'ath and Szab\'o prove that the map \[F_0: HF^\infty(Y, \mathfrak s_0; \mathbb F_2) \to HF^\infty(Y_{-n}(K), \mathfrak t_0; \mathbb F_2)\] is an isomorphism, and as discussed in the proof of Lemma \ref{lemma:first-action} this commutes with the action by diffeomorphisms. It follows that $\psi_*$ only depends on the isotopy class in $\Sigma_g \times S^1$ for this term as well.
\end{proof}

Now we are reduced to algebraic considerations about the surgery triangle. Recall that the map $F$ in the exact sequence is a sum of two terms $F_0$ and $F_1$. Both $F_0$ and $F_1$ are $\text{MCG}^{++}(\Sigma_g \times S^1)$-equivariant isomorphisms, and $F_1 = F_0 J$, where $J: (HF^\infty)^*(Y, \mathfrak s_0) \to (HF^\infty)^*(Y, \mathfrak s_0)$ is the `basepoint-swapping automorphism'. Finally, with respect to the isomorphism \eqref{eq:startingpoint}, \cite{JabukaMark} compute 
\[J(e^{-\omega U} \wedge (x U^i)) = e^{-\omega U} \wedge -\iota_{e^{-\omega U}}(x U^i).\] 

In particular, with respect to the further isomorphism given by wedging by $e^{-\omega U}$, the map $J$ is given by contraction against $-e^{-\omega U}$. Notice that $J$ is equivariant under the action of the mapping class group $\text{MCG}^{++}(\Sigma_g \times S^1)$.

The signs are irrelevant; the action of $-e^{-\omega U}$ can be taken to $e^{\omega U}$ by composing in the first $\Lambda^*$ term with the isomorphism $\Lambda^*(\mathbb Z^{2g})[U, U^{-1}]$ which is $+1$ on $\Lambda^p$ for $p \equiv 0,1 \mod 4$ and $-1$ on $\Lambda^p$ for $p \equiv 3, 4 \mod 4$ and in the second $\Lambda^*$ term with the negative of this isomorphism. Applying this isomorphism and wedging with $e^{-\omega U}$, we see that the surgery long exact sequence is isomorphic to the long exact sequence 
\[\cdots \to HF^\infty(\Sigma_g \times S^1, \mathfrak s_0)\to \Lambda^*(\mathbb Z^{2g})[U, U^{-1}] \xrightarrow{\iota_{e^{\omega U}-1}} \Lambda^*(\mathbb Z^{2g})[U, U^{-1}] \to HF^\infty(\Sigma_g \times S^1, \mathfrak s_0) \to \cdots\]

This can then be reduced to an $\text{MCG}^{++}(\Sigma_g \times S^1)$-equivariant short exact sequence 
\begin{equation}\label{eqn:HF-SES}
0 \to \coker(e^{\omega U} - 1;\mathbb F_2) \to HF^\infty(\Sigma_g \times S^1, \mathfrak s_0) \to \ker(e^{\omega U} - 1; \mathbb F_2) \to 0;
\end{equation}
here, recall that we write $\ker(\alpha)$ and $\coker(\alpha)$ to denote the kernel and cokernel of \emph{contraction} by $\alpha$. 

This completes the proof of Proposition \ref{prop:HF-MCG-calc} over $\mathbb F_2$. If the Heegaard Floer machinery lifts to $\mathbb Z$, so does the proof above, barring the restriction of Lemma \ref{lemma:compute-on-Y} to $g \ge 3$.

\subsection{Proofs of the main theorems}

\begin{proof}[Proof of Theorem \ref{thm:HFHC-filt-intro}]
If we are not keeping track of mapping class group actions, the short exact sequence \eqref{eqn:HF-SES} was constructed over $\mathbb Z$ in \cite[Section 4.3]{JabukaMark} without additional assumptions on the foundations of Heegaard Floer homology. Because $\ker(\omega; \mathbb Z)$ and $\ker(e^{\omega U} - 1; \mathbb Z)$ are free $\mathbb Z[U, U^{-1}]$-modules, the short exact sequence \eqref{eqn:HF-SES} splits and we have isomorphisms of $\mathbb Z[U, U^{-1}]$-modules
\begin{align*}
    HF^\infty(\Sigma_g \times S^1, \mathfrak s_0; \mathbb Z) &\cong \coker(e^{\omega U} - 1;\mathbb Z) \oplus \ker(e^{\omega U} - 1; \mathbb Z) \\
    HC_*(\Sigma_g \times S^1; \mathbb Z) &\cong \coker(\omega; \mathbb Z) \oplus \ker(\omega; \mathbb Z).
\end{align*}
That these two $\mathbb Z[U, U^{-1}]$-modules are isomorphic follows from Lemma \ref{lemma:ker-contraction} and Proposition \ref{prop:computation-for-coker}(a). The appearance of $U$ is inconsequential; its primary function is to make $e^{\omega U} - 1$ and $\omega U$ homogeneous maps of degree zero. The results of Proposition \ref{prop:computation-for-coker}(a) and (b) may be easily modified to include the action of $U$. \\

We now take $g \ge 3$. Using the filtrations of Proposition \ref{prop:computation-for-coker}(b) and the exact sequences of Propositions \ref{prop:HC-MCG-calc} and \ref{prop:HF-MCG-calc}, we may filter $HF^\infty$ and $HC_*$ by 
\begin{align*}
F_r HF^\infty(\Sigma_g \times S^1, \mathfrak s_0; \mathbb Z) &= \begin{cases} F_k \coker(e^{\omega U} - 1; \mathbb Z) & 0 \le k \le g \\
HF^\infty(\Sigma_g \times S^1; \mathbb Z) & k > g \end{cases} \\
F_r HC_*(\Sigma_g \times S^1; \mathbb Z) &= \begin{cases} F_k \coker(\omega; \mathbb Z) & \;\;\;\;\; 0 \le k \le g \\
HC_*(\Sigma_g \times S^1; \mathbb Z) & \;\;\;\;\;k > g \end{cases}
\end{align*}
The computation of associated graded modules stated in Theorem \ref{thm:HFHC-filt-intro}(b) now follows from Lemma \ref{lemma:ker-contraction} and Proposition \ref{prop:computation-for-coker}(b).\end{proof}

\begin{remark}\label{rmk:F2-case}
The proofs of both Lemma \ref{lemma:ker-contraction} and Lemma \ref{lemma:specseq-collapse} fail over $\mathbb F_2$, and it is not clear what one should expect about their cokernels.

Over $\mathbb F_2$, one instead finds that \begin{align*}
        \text{gr}_{k+1} HF^\infty(\Sigma_g \times S^1, \mathfrak s_0; \mathbb F_2) &\cong \ker(e^{\omega U} - 1; \mathbb F_2) \\
        \text{gr}_{k+1} HC_*(\Sigma_g \times S^1; \mathbb F_2) &\cong \ker(\omega; \mathbb F_2).
    \end{align*}
The authors have not attempted to compare these modules, but it seems likely they fail to be isomorphic for sufficiently large $g$. 
\end{remark}

\begin{proof}[Proof of Theorem \ref{thm:no-iso-intro}]
The key is to use the action of transvections to reduce to the cokernel term. 

\begin{lemma}\label{lemma:contraction1}
For any commutative ring $R$ and any $g \ge 1$, if $\alpha \in \ker(\omega;R)$ is nonzero, there exists $f \in H^1(\Sigma_g; \mathbb Z)$ so that $[\alpha \wedge f] \ne 0 \in\coker(\omega;R).$
\end{lemma}
\begin{proof}
The proof is by induction on $g$. The case $g = 1$ is straightforward, as there $\ker(\omega;R) = \Lambda^0(R^2) \oplus \Lambda^1(R^2)$ while $\coker(\omega;R) = \Lambda^1(R^2) \oplus \Lambda^2(R^2)$, and the wedge product pairing on $\Lambda^1(R^2)$ is nondegenerate. 

Write $\omega$ for the symplectic form on $R^{2g+2}$ and $\eta$ for the symplectic form on $R^{2g}$. If $\alpha \in \ker(\omega;R)$ and $e^i \alpha \in \text{im}(\iota_\omega)$ for all $i$, we aim to show that $\alpha = 0$. We will do so by showing that for every pair of indices $\{2i-1,2i\}$, that $\alpha$ consists of monomials disjoint from this pair. To simplify notation, we present this argument specifically for the pair $\{2g+1, 2g+2\}$. Taking an arbitrary $\alpha \in \Lambda^*(R^{2g+2}),$ we compute 
\begin{align*}\alpha &= a + e^{2g+1} b + e^{2g+2} c + e^{2g+1} e^{2g+2} d, \quad a,b,c,d \in \Lambda^*(R^{2g}), \\
\iota_\omega(\alpha) &= \iota_\eta(a) - d + e^{2g+1} \iota_\eta(b) + e^{2g+2} \iota_\eta(c) + e^{2g+1} e^{2g+2} \iota_\eta(d).
\end{align*}
If $\alpha \in \ker(\omega;R)$, then $b, c, d \in \ker(\eta;R)$. By inductive hypothesis, if $d \ne 0$, then there exists $0 \le i \le g$ for which $e^i d \not\in \text{im}(\iota_\eta)$; the formula above then implies $e^i \alpha \not\in \text{im}(\iota_\omega).$ A similar argument implies if $b$ or $c$ is nonzero, so $\alpha = a$ and the monomials appearing in $\alpha$ are disjoint from $\{2g+1, 2g+2\}$. Applying the same argument to all other coordinates, it follows that $\alpha = 0$. 
\end{proof}

We will now prove Theorem \ref{thm:no-iso-intro} over $\mathbb Z$, as the argument over $\mathbb F_2$ just removes a reduction step. Define the finitely generated abelian groups \[HC_{U=1} = HC \otimes_{\mathbb Z[U]} \mathbb Z, \quad \quad HC_{U=1}^T = \{x \in HC_{U=1} \mid f \cdot x = x \text{ for all } f \in H^1(\Sigma_g; \mathbb Z)\}\] and similarly for $HF$. It follows from Proposition \ref{prop:HC-MCG-calc} and Lemma \ref{lemma:contraction1} that $HF^T = \coker(\omega;\mathbb Z)$. 

Suppose, towards a contradiction, that we had an $\text{MCG}^{++}(\Sigma_g \times S^1)$-equivariant  isomorphism $HC \cong HF$ of $\mathbb Z[U]$-modules. This induces an isomorphism $HC_{U=1} \cong HF_{U=1}$, which restricts to an isomorphism $HC_{U=1}^T \cong HF_{U=1}^T$ of their fixed point subgroups. These subgroups retain an action of $\text{MCG}^+(\Sigma_g)$, and the isomorphism necessarily preserves this action. 

If we can show that $HF_{U=1}^T = \coker(e^{\omega} - 1; \mathbb F_2)$ with the action described in Proposition \ref{prop:HF-MCG-calc}, then the desired result follows from Proposition \ref{prop:computation-for-coker}(c). We have a containment 
\[\coker(e^\omega - 1;\mathbb F_2) \subset HF_{U=1}^T \cong \coker(\omega; \mathbb F_2);\]
to prove equality, it suffices to show these spaces have the same dimension.  Proposition \ref{prop:computation-for-coker}(a) gives an isomorphism $\coker(e^\omega - 1; \mathbb Z) \cong \coker(\omega; \mathbb Z).$ Because tensor products are right exact, we obtain \[\dim \coker(e^\omega-1;\mathbb F_2) = \dim \coker(e^\omega-1;\mathbb Z)/2 = \dim \coker(\omega;\mathbb Z)/2 = \dim \coker(\omega; \mathbb F_2).\qedhere\]
\end{proof}



\bibliography{biblio.bib}
\bibliographystyle{alphaurl} 
\end{document}